\documentclass{amsart}
\usepackage[utf8]{inputenc}
\usepackage[T1]{fontenc}
\usepackage[english]{babel}
\usepackage{textcomp}
\usepackage{lmodern}
\usepackage{amsthm}
\usepackage{amsmath}
\usepackage{amssymb}
\usepackage{mathrsfs}
\usepackage{enumerate}
\usepackage{dsfont}
\usepackage[titletoc]{appendix}
\usepackage[left= 2.5cm,right=2.5cm,top=3.5cm,bottom=3.5cm]{geometry}
\newtheorem{theo}{Theorem}[section]
\newtheorem{prop}[theo]{Proposition}

\newtheorem{lem}[theo]{Lemma}

\theoremstyle{remark}
\newtheorem{rema}[theo]{Remark}

\theoremstyle{definition}
\newtheorem{defi}[theo]{Definition}

\makeatletter
\@addtoreset{equation}{section}

\makeatother

\title{On massless electron limit for a multispecies kinetic system with external magnetic field}

\begin{document}

\author{Maxime Herda
\email{herda@math.univ-lyon1.fr}
\address{Université Claude Bernard Lyon 1,
CNRS UMR 5208,
Institut Camille Jordan,
43 blvd. du 11 novembre 1918,
F-69622 Villeurbanne cedex,
France}}

\begin{abstract}
 We consider a three-dimensional kinetic model for a two species plasma consisting of electrons and ions confined by an external nonconstant magnetic field.
 Then we derive a kinetic-fluid model when the mass ratio $m_e/m_i$ tends to zero. 
 
 Each species initially obeys a Vlasov-type equation and the electrostatic coupling follows from a Poisson equation. 
 In our modeling, ions are assumed non-collisional while a Fokker-Planck collision operator is taken into account in the electron equation. 
 As the mass ratio tends to zero we show convergence to a new system where the macroscopic electron density satisfies an anisotropic drift-diffusion equation. 
 To achieve this task, we overcome some specific technical issues of our model such as the strong effect of the magnetic field on electrons and the lack of regularity at the limit.
 With methods including renormalized solutions, relative entropy dissipation and velocity averages, we establish the rigorous derivation of the limit model.

 \smallskip
\noindent \textsc{Keywords.} Multispecies; plasma physics; magnetic field; Vlasov-Poisson; Vlasov-Poisson-Fokker-Planck; drift-diffusion.

 \smallskip
\noindent \textsc{AMS subject classifications.} 35Q83, 35Q84, 35B25, 82D10, 76X05
\end{abstract}


\maketitle
\tableofcontents
\section{Introduction}
\label{intro}

In many plasma physics applications, the numerical simulation of full multi-species kinetic systems of equations can be extremely expensive in computer time.
Indeed, since the typical time, space and velocity scales of each species differ from several orders of magnitude, it requires a fine discretization to accurately approximate the different scales. 
We refer to \cite{badsi_2015_modelling} for a discussion on these issues in the two-species Vlasov-Poisson case. 
Therefore, part of the problem is sometimes overcome by making simplifying assumptions on species with negligible contribution to the whole dynamic. 
In this paper, we are interested in reducing a kinetic model by taking the limit when the mass ratio between light and heavy particles, namely electrons and ions, tends to $0$.

The charged gas evolves under its self-consistent electrostatic field and an external magnetic field. 
This configuration is typical of a tokamak plasma \cite{bellan_2006_fundamentals, miyamoto_2006_plasma} where the magnetic field is used to confine particles inside the core of the device.
We assume that on the time scale we consider, collisions on ions can be neglected while for electrons, it is entirely modeled with a Fokker-Planck operator.
At the formal level, an exhaustive study of asymptotic mass-disparate model with more involved collision operators such as the Boltzmann or Landau operator can be found in the review \cite{degond_2007_asymptotic} of Degond.
When the mass ratio goes to $0$, our model converges to a drift-diffusion equation featuring a magnetic-field dependent diffusion matrix for the electrons coupled with the original kinetic equation for the heavy particles. 
Similar parabolic equations with non-symmetric diffusion for plasmas can also be found in \cite{ben_2008_diffusion, degond_1998_model, degond_2001_diffusion, degond_2002_diffusion, degond_2007_asymptotic}.

Hereafter, we start from the physical equations and propose a detailed scaling with respect to the ions time scale. 
In the dimensionless system, for the light species equation, the leading order terms with respect to the mass ratio are those related to the magnetic field and the collisions.
Therefore, the resulting derivation is in the mean time similar to a strong magnetic field limit as in the papers of Golse and Saint-Raymond \cite{golse_1999_vlasov, saint_2000_gyrokinetic},
and to a diffusive or parabolic limit as in the work of Poupaud, Soler, Masmoudi and El Ghani in \cite{poupaud_2000_parabolic, el_2010_diffusion}.
While the latter papers provide us with many tools to handle our own problem, some technical issues in our analysis are closely related to the special features of our model.
In a single-species model of charged particles, the other particles density is usually given either as a static regular background or as a function of the electric potential,
while our ions are only known to obey a non-trivial kinetic equation. 
Because of this coupling, it turns out that some extra analysis is needed to recover the regularity required for our limit system to make sense.
Besides, we have to control the strong magnetic field. This is done by looking at the interplay between the increasing effect of oscillations and collisions in the asymptotic regime where the mass ratio tends to 0.
Indeed, despite the fact that collisions and magnetic forces appear at the same order of magnitude in original and limit equations we consider, we prove, by making the most of special cancellations, that with other respects collisional effects provide a form of control on magnetic contributions.
While the rest of our analysis seems rather robust the latter control --- that relies on cancellations --- does not seem to extend readily to other forms of dissipation operators such as linear Boltzmann collision operators.

\subsection{The physical model}

%
To avoid unnecessary technicalities, we suppose that there is only one type of ion in the plasma of mass $m_i$ and charge $q$, while $m_e$ denotes the electron mass and $-q$ their negative charge. 
The case of a plasma containing several types of ions can be treated in the same way.

The number of particle of type $\alpha$ in the phase space volume $dx\ dv$ at position $x\in\mathbb{R}^3$, with velocity $v\in\mathbb{R}^3$ between time $t$ and $t+dt$ is $f_\alpha(t,x,v)\ dx\ dv\ dt$.
The index $\alpha$ stands for the species of particles and can be either $i$ for ions or $e$ for electrons. 
The ion distribution function $f_i$ evolves according to a Vlasov equation and the electron distribution function $f_e$ follows a Vlasov-Fokker-Planck equation.
The coupling occurs through the Poisson equation that relates the electric field to the densities. 
The equations of the model, written in physical units, are the following
\begin{equation}
\left\{
 \begin{aligned}
&\partial_tf_i + v\cdot\nabla_xf_i  + \frac{q}{m_i}\left(- \nabla_x\phi + v\wedge B_\text{ext} \right) \cdot \nabla_vf_i = 0,\\
&\partial_tf_e + v\cdot\nabla_xf_e  - \frac{q}{m_e}\left(- \nabla_x\phi + v\wedge B_\text{ext} \right) \cdot \nabla_vf_e = \frac{1}{t_\text{col}}\nabla_v\cdot ( vf_e + \frac{k_B \theta}{m_e}\nabla_vf_e),\\
&-\varepsilon_0\Delta_x\phi = q(n_i - n_e),
\end{aligned}
\right.
\label{eqphy}
\end{equation}
where $t_\text{col}$ is the characteristic time between two collisions, $k_B$ is the Boltzmann constant, $\theta$ is the average electron temperature and $\varepsilon_0$ is the dielectric constant. 
The Poisson equation involves the macroscopic densities
\[
 n_{\alpha} = \int_{\mathbb{R}^3}f_{\alpha}dv, \quad\forall\alpha\in\{i,e\}.
\]

As mentioned before the typical time scales of ions and electrons largely differ due to the smallness of the mass ratio $m_e/m_i$.
In applications involving magnetic confinement fusion, ions are the particles of interest and approximations are made on the electron distribution function $f_e$ in order to simplify the model. 
A common reduction supposes that the macroscopic electron density is given by the Maxwell-Boltzmann density 
\[
 n_\text{MB}(t,x) = C(t) e^{\frac{q\phi(t,x)}{k_B \theta}}, 
\]
where $C(t)$ is a normalization function. 
The derivation of the latter from \eqref{eqphy}, with $B_\text{ext}=0$, is discussed in \cite{bouchut_1991_global} and obtained in \cite{bouchut_1995_on} for a one species Vlasov-Poisson-Fokker-Planck model.
To our knowledge, the case of a magnetized plasma has never been treated before. 
We stress that the presence of a strong magnetic field modifies even formal computations.

The goal of the present paper is to derive a fluid equation for electrons, when the mass ratio tends to $0$. 
The evolution of their dynamic shall obey an equation on the macroscopic density.
In the next paragraph we write the system in a consistent dimensionless form to receive the fluid model in the asymptotic regime of massless electrons.

\subsection{Scaling}

We denote by $L$ the characteristic length of the system, $t_0$ the characteristic time and $V_\alpha$ the thermal velocity for the species $\alpha\in\{i,e\}$. 
For any other physical quantity $G$, we denote by $\bar{G}$ the characteristic value of $G$ and $G'$ the dimensionless quantity associated to $G$ so that $G = \bar{G}G'$.
We assume that the plasma is globally neutral, which means that \[\bar{n}_i = \bar{n}_e =: N,\] and that the characteristic temperature (or kinetic energy) of each species are equal. 
A plasma satisfying the latter hypothesis is called a hot plasma \cite{bellan_2006_fundamentals} and satisfies with our notation \[m_iV_i^2 = m_eV_e^2 = k_B\theta.\]
The new unknowns of the system are then defined by the following relations
\[
 f_\alpha(t,x,v) = \frac{N}{V_\alpha^3}f'_\alpha\left(\frac{t}{t_0}, \frac{x}{L}, \frac{v}{V_\alpha}\right),\quad n_\alpha(t,x) = N n_\alpha'\left(\frac{t}{t_0}, \frac{x}{L}\right),\]\[  \phi(t,x) = \bar{\phi}\phi'\left(\frac{t}{t_0}, \frac{x}{L}\right).
\]

The dimensional analysis of \eqref{eqphy} introduces several important physical constants of the system, namely, for each species $\alpha\in\{i,e\}$
\[
 \lambda_D = \sqrt{\frac{\varepsilon_0k_B \theta}{q^2N}},\quad t^{(\alpha)}_{p} = \frac{\lambda_D}{V_\alpha},\quad t^{(\alpha)}_{c} = \frac{m_\alpha}{q\bar{B}_\text{ext}},\quad r_L^{(\alpha)} = V_\alpha t_c^{(\alpha)},
\]
which are respectively the Debye length, the plasma time, the cyclotron time and the Larmor radius.
More details on these constants can be found in the physics literature (see \cite{bellan_2006_fundamentals,miyamoto_2006_plasma}). 
We mention that the first two are typical scales of the electrostatic effects while the last two  are related to magnetic phenomena. 
Since the goal is to perform a model reduction for the electron dynamic, we choose a scaling relative to the typical ion time scale.
In particular, it means that we choose
\[t_0 = \frac{L}{V_i}.\]
From some of the original physical constants and characteristic quantities arise dimensionless parameters of the system, namely
\[
 \delta = \frac{\lambda_D}{L},\quad \eta = \frac{q\bar{\phi}}{k_B \theta},\quad \mu = \frac{\bar{\phi}}{LV_i\bar{B}_\text{ext}},\quad \varepsilon = \sqrt{\frac{m_e}{m_i}}.
\]
In a tokamak plasma these parameters would all be small.
 Dealing with the asymptotic $\delta\rightarrow0$ is called the quasineutral limit and has been first investigated for the Vlasov-Poisson system by Brenier and Grenier in
\cite{limite_1994_brenier, oscillations_1996_grenier, brenier_2000_convergence} and developed recently by Han-Kwan, Hauray and Rousset in \cite{han_2011_quasineutral, han_2015_quasineutral, han_2015_stability}.
The second parameter $\eta$ is called the coupling parameter, for it measures the importance of the electrostatic effects with respect to the thermal agitation.
The third parameter $\mu$ compares the Coulomb (electric) and the Laplace (magnetic) forces. 
The limit $\mu\rightarrow 0$ corresponds to the case of a strong magnetic field and has been studied for single-species plasma by Frenod and Sonnendrucker in \cite{frenod_1998_homogenization} 
and Golse and Saint-Raymond in \cite{golse_1999_vlasov, saint_2000_gyrokinetic}. 
This asymptotic is called the gyrokinetic or drift-kinetic approximation in the mathematics literature.
The last parameter $\varepsilon$, quantifying the mass ratio, is the main concern of this paper. 
We are interested in the limit of massless electrons, namely $\varepsilon\rightarrow 0$. 
The dimensionless equations write  
\begin{equation}
\left\{
 \begin{aligned}
&\partial_tf'_i + v\cdot\nabla_xf'_i  - \eta\nabla_x\phi'\cdot \nabla_v f'_i + \frac{\eta}{\mu}(v\wedge B_\text{ext}) \cdot \nabla_vf'_i = 0,\\
&\partial_tf'_e + \frac{1}{\varepsilon}v\cdot\nabla_xf'_e  + \frac{\eta}{\varepsilon}\nabla_x\phi'\cdot \nabla_v f'_e - \frac{\eta}{\mu\varepsilon^2}(v\wedge B_\text{ext}) \cdot \nabla_vf'_e = \frac{t_0}{t_\text{col}}\nabla_v\cdot ( vf'_e + \nabla_vf'_e)\\
&-\eta\delta^2\Delta_x\phi^\varepsilon =  n'_i -  n'_e.
\end{aligned}
\right.
\end{equation}

The link between the different time scales and dimensionless parameters follows from the previous relations and writes
\[
  t^{i}_{p} =\frac{1}{\varepsilon^2}t^{e}_{p} = \delta t_0,\quad t^{i}_{c} =\frac{1}{\varepsilon}t^{e}_{c} = \frac{\mu}{\eta}t_0.
\]
Finally we consider the regime where
\[
 \frac{t_0}{t_\text{col}} = \frac{1}{\varepsilon^2}.
\]
which means that the ratio between collision time and the observation time is of the same order than the mass ratio.
 This can be derived considering that the Knudsen number, namely the ratio between the mean free path of the particles and the observation 
length $L$ is of order $\varepsilon$. More details on this part of the scaling can be found in the paper
of the physicists Petit and Darrozes \cite{petit_1975_new} and in the work of Degond and Lucquin \cite{degond_1996_asymptotics, degond_1996_transport, degond_2007_asymptotic}
and more recently Graille, Magin, Massot and Giovangigli \cite{graille_2009_kinetic, giovangigli_2010_multicomponent, giovangigli_2010_multicomponent2}.
This makes the collision the leading order term with the magnetic field in the electron equation. 
In its analysis, the resulting scaling slightly differs from the usual parabolic or hydrodynamic scaling of Vlasov-Poisson-Fokker-Planck
since the magnetic field introduces fast oscillations when $\varepsilon$ tends to zero and it is \emph{a priori} not clear that the latter
can be controlled by the dissipation effect due to collisions.

By taking weaker collisions, say $t_0/t_\text{col} = 1/\varepsilon$, one formally derives a Maxwell-Boltzmann density in the direction parallel to the magnetic field
and a guiding center model in the perpendicular plane when $\varepsilon\rightarrow 0$.
This limit model features interesting and well-known physical phenomena. Its derivation and analysis will be investigated in future work.

\subsection{The mathematical model}
From now on, we keep $\delta, \eta, \mu$ fixed and thus, for writing convenience, we assume that $\delta = \eta = \mu = 1$. 
The rescaled two species (Vlasov)-(Vlasov-Fokker-Planck)-(Poisson) kinetic system with external magnetic field writes
\begin{equation}
\left\{
\begin{aligned}
&\partial_tf^{\varepsilon}_i + v\cdot\nabla_xf^{\varepsilon}_i  - \nabla_x\phi^{\varepsilon}\cdot \nabla_v f^{\varepsilon}_i + (v\wedge B) \cdot \nabla_vf^{\varepsilon}_i = 0,\\
&\varepsilon\partial_tf^{\varepsilon}_e + v\cdot\nabla_xf^{\varepsilon}_e  + \nabla_x\phi^{\varepsilon}\cdot \nabla_v f^{\varepsilon}_e - \frac{1}{\varepsilon}(v\wedge B) \cdot \nabla_vf^{\varepsilon}_e = \frac{1}{\varepsilon}\mathcal{L}_{I}(f^{\varepsilon}_e),\\
&-\Delta_x\phi^\varepsilon =  n^{\varepsilon}_i -  n^{\varepsilon}_e.
\end{aligned}
\right.
\label{Vi-VFPe-P}
\end{equation}
where we have dropped primes and marked the dependency on $\varepsilon$ with an exponent.
The Fokker-Planck operator on the right-hand side of the second equation of \eqref{Vi-VFPe-P} is defined by
\[
 \mathcal{L}_{I}: f\mapsto \nabla_v\cdot (vf + \nabla_vf).
\]
The external magnetic field is of the form
\[
B(t,x) = \left(b_1\ b_2\ b_3\right)^\top,
\]
in the canonical euclidean basis of $\mathbb{R}^3$. 
The magnetic field $B$ is assumed to belong to $L^\infty(0,T;L^{\infty}_x)$ and to satisfy the Maxwell-Thompson equation $\nabla_x\cdot B = 0$. 
Its effect on the particle is given by the Lorentz force, appearing in the kinetic equations as ''$(v\wedge B)\cdot\nabla_v$''.
The Cauchy problem is completed with initial conditions on distribution functions
\[
f^\varepsilon_{\alpha}(t=0,x,v) = f^\text{in}_{\alpha}(x,v), \quad\forall\alpha\in\{i,e\}.
\]
The Poisson equation can be reformulated using its fundamental solution $\Phi$, in dimension $3$
\begin{equation}
 \Phi(x) = \frac{1}{4\pi|x|}.
 \label{Phi}
\end{equation}
The potential $\phi^\varepsilon$ is then given by
\begin{equation}
 \phi^\varepsilon(t,x) = \Phi\ast_x n_\varepsilon = \int_{\mathbb{R}^3}\Phi(x-y) n^\varepsilon(t,y) dy,
 \label{elecpot}
\end{equation}
where $n^\varepsilon$ denotes the total charge density and is given by
\[ 
n^\varepsilon(t,x) =  n^\varepsilon_i(t,x) -  n^\varepsilon_e(t,x).
\]
For later use, we also introduce the current density
 \[
 j^\varepsilon(t,x) = j^\varepsilon_i(t,x) - j^\varepsilon_e(t,x)\text{, where }j_i = \int_{\mathbb{R}^3}vf^\varepsilon_i dv\text{ and }j_e = \frac{1}{\varepsilon}\int_{\mathbb{R}^3}vf^\varepsilon_e dv.
 \]
and the rescaled uniform Maxwellian, in dimension 3
\begin{equation}
 M(v) = \frac{1}{(2\pi)^{3/2}}e^{-\frac{|v|^2}{2}}.
 \label{maxwellian}
\end{equation}

The existence of global weak solutions for the Vlasov-Poisson system, namely the first and the last equation of \eqref{Vi-VFPe-P} with a given electron background $n_e^{\varepsilon}$,
was first established by Arsenev in \cite{arsenev_1975_existence}. 
DiPerna and Lions improved conditions on the initial data for this result in  \cite{diperna_1988_global} and adapted their renormalization techniques for the Vlasov-Poisson system \cite{di_1988_solutions}.
Concerning classical solutions, global existence in dimension $2$ was proved by Ukai and Okabe for well-localized initial data in \cite{ukai_1978_classical}.
The case of dimension $3$ was considered by Bardos and Degond in \cite{bardos_1985_global}. 

The theory for the Vlasov-Poisson-Fokker-Planck system, consisting of last two equations of \eqref{Vi-VFPe-P} with a given ion background $n_i^{\varepsilon}$, is also well-known.
The existence of global weak solutions is proved in \cite{victory_1991_existence, carrillo_1995_initial} and results on the existence and uniqueness of smooth solutions were 
first obtained by Degond in \cite{degond_1986_global} and generalized by Bouchut in \cite{bouchut_1993_existence}. 
Here, we shall consider DiPerna-Lions renormalized solutions of the Vlasov-Fokker-Planck equation as introduced in \cite{di_1988_solutions}. 
Our choice is motivated by two observations. 
First, the presence of the friction part of the Fokker-Planck operator precludes any hope for a uniform (in time and in $\varepsilon$) $L^p_{x,v}$ bound of $f_e^\varepsilon$ in the small $\varepsilon$ limit when $p>1$. 
As Poupaud and Soler proved in \cite{poupaud_2000_parabolic}, a uniform in $\varepsilon$ bound can be obtained for the adequate weighted $L^p$ norm but only on a finite time decreasing with $p$, which is not quite satisfying.
On the other hand, knowing merely that  $f_e^\varepsilon$ lies in $L^1_{x,v}$ and that $\nabla_x\phi^\varepsilon$ belongs to $L^2_x$ is insufficient to give even a distributional
sense to the product $f_e^\varepsilon\nabla_x\phi^\varepsilon$ appearing in the electron equation of \eqref{Vi-VFPe-P}. 
The concept of renormalized solutions proposes to replace the direct consideration of the troublesome equation for $f_e^\varepsilon$ by a family of meaningful equations for $\beta(f_e^\varepsilon)$,
$\beta$ ranging through a sufficiently large class of smooth functions.

For the coupled system \eqref{Vi-VFPe-P}, we consider weak solutions of the ion equation and renormalized solutions of the electron equation, the electric potential being given by \eqref{elecpot}. 

\begin{defi}
 We say that a triplet $(f^\varepsilon_i, f^\varepsilon_e, \phi^\varepsilon)$ is a solution on $[0,T)$ to the Cauchy problem \eqref{Vi-VFPe-P} with initial data $f^\text{in}_i, f^\text{in}_e$ if it satisfies
 \begin{enumerate}
  \item[\textbf{1)}] $f^\varepsilon_i\in L^\infty(0,T;L^1_{x,v}\cap L^\infty_{x,v})$; $f^\varepsilon_e\in L^\infty(0,T;L^1_{x,v})$; $f^\varepsilon_{\alpha}\geq 0$ almost everywhere for $\alpha\in\{i,e\}$,
  \item[\textbf{2)}] $\phi^\varepsilon \in L^\infty(0,T;\dot{H}^1_x)$; $\nabla_v\sqrt{f^\varepsilon_e}\in L^2(0,T;L^2_{x, v})$,
\end{enumerate}
  and, for every $\varphi\in \mathcal{D}([0,T)\times\mathbb{R}^6)$
 \begin{enumerate}
  \item[\textbf{3)}] The mapping $t\mapsto\iint_{\mathbb{R}^6}\varphi f^\varepsilon_i dv dx$ is continuous and the first equation of \eqref{Vi-VFPe-P} holds in the sense of distributions on $[0,T)\times\mathbb{R}^6$ 
  with the initial condition $f_i^\text{in}$.
  \item[\textbf{4)}] For every function $\beta\in\mathcal{C}^2(\mathbb{R}_+)$ satisfying \[|\beta(u)|\leq C(\sqrt{u} + 1),\ |\sqrt{u}\beta'(u)|\leq C,\ |u\beta''(u)|\leq C,\] for some $C>0$, 
  the mapping $t\mapsto\iint_{\mathbb{R}^6}\varphi \beta(f^\varepsilon_e) dv dx$ is continuous and $\beta(f^\varepsilon_e)$ satisfies in the sense of distributions on $[0,T)\times\mathbb{R}^6$ 
  \begin{equation}
   \varepsilon\partial_t\beta(f^\varepsilon_e) + v\cdot\nabla_x \beta(f^\varepsilon_e) + \nabla_x\phi^\varepsilon\cdot\nabla_v\beta(f^\varepsilon_e) - \frac{1}{\varepsilon}(v\wedge B)\cdot\nabla_v\beta(f^\varepsilon_e) = 
   \frac{1}{\varepsilon}\mathcal{L}_{I}(f^\varepsilon_e)\beta'(f^\varepsilon_e),
   \label{ebeta}
  \end{equation}
  with the initial condition $\beta(f_e^\text{in})$.
  \item[\textbf{5)}] For any $\lambda>0$, $\theta_{\varepsilon,\lambda} = \sqrt{f^\varepsilon_e+\lambda M}$ is such that the mapping $t\mapsto\iint_{\mathbb{R}^6}\varphi \theta_{\varepsilon,\lambda} dv dx$ is continuous 
  and $\theta_{\varepsilon,\lambda}$ satisfies in the sense of distributions on $[0,T)\times\mathbb{R}^6$
  \begin{equation}
   \varepsilon\partial_t\theta_{\varepsilon,\lambda} + v\cdot\nabla_x\theta_{\varepsilon,\lambda} + \nabla_x\phi^\varepsilon\cdot\nabla_v\theta_{\varepsilon,\lambda}  =
   \frac{1}{2\varepsilon\theta_{\varepsilon,\lambda}}\mathcal{L}_{A}(f^\varepsilon_e) - \frac{\lambda M}{2\theta_{\varepsilon,\lambda}}v\cdot\nabla_x\phi^\varepsilon,
   \label{etheta}
  \end{equation}
 where the Fokker-Planck operator and the term related to the magnetic field are gathered within the operator $\mathcal{L}_{A}$ which may be written in the following form
 \begin{equation}
  \mathcal{L}_{A}(f) = \nabla_v\cdot(A(t,x)vf + \nabla_vf)
 \end{equation}
 where the matrix $A(t,x)$ is given by
 \begin{equation}
 Av = v+v\wedge B = \left(\begin{matrix}1&b_3&-b_2\\-b_3&1&b_1\\b_2&-b_1&1\end{matrix}\right)v.
 \end{equation}
 \end{enumerate}
 \label{D1}
\end{defi}

\begin{rema}
 One readily checks that weak formulations of 3), 4) and 5) are consistent with estimates in 1) and 2). 
 We stress that we need to use two types of renormalization for the electron Vlasov-Fokker-Planck equation.
 Actually, points 1) to 4) are sufficient to define a self-consistent notion of solution for which we can prove an existence result. 
 The introduction of the additional equation $\eqref{etheta}$ is inspired by \cite{el_2010_diffusion}.
 It comes from a renormalization of the equation satisfied by $f_e^\varepsilon/M$ with the function $s\mapsto\sqrt{s+\lambda}$.
 While equation \eqref{ebeta} provides us with the required alternative meaning for the electron equation in \eqref{Vi-VFPe-P}, we shall pass to the limit $\varepsilon\rightarrow 0$ in \eqref{etheta}.
\end{rema}

\subsection{Main result}

Our main goal is to prove the convergence of solutions to \eqref{Vi-VFPe-P}--in the sense of Definition \ref{D1}--towards weak solutions of the following coupled kinetic-fluid system
 \begin{equation}
 \left\{
 \begin{aligned}
  &\partial_tf_i + v\cdot\nabla_xf_i - \nabla_x\phi\cdot\nabla_vf_i + (v\wedge B)\cdot\nabla_vf_i = 0,\\
  &\partial_t n_e + \nabla_x\cdot j_e = 0,\\
  & j_e = -D(\nabla_x n_e - \nabla_x\phi n_e),\\
  &-\Delta_x\phi =  n_i -  n_e,\\
  &f_i(0,\cdot,\cdot) = f_i^\text{in}\text{ and } n_e(0,\cdot,\cdot) = \int f_e^\text{in}dv.
  \end{aligned}
  \right.
  \label{driftdiff}
 \end{equation}
where the diffusion matrix is given by
\[
D(t,x) = A^{-1}(t,x) = \frac{1}{1+|B|^2}\left(\begin{matrix}1 + b_1^2&-b_3+b_1b_2&b_2 + b_1b_3\\b_3 + b_1b_2&1 + b_2^2&-b_1+b_2b_3\\-b_2 + b_1b_3&b_1 + b_2b_3&1+b_3^2\end{matrix}\right).
\]
 Since $B$ is essentially bounded, it yields \[D \in~L^\infty(0,T;L^{\infty}_x).\] Moreover by denoting $I_3$ the identity matrix in dimension $3$, one sees that
\[
 \text{Re}(D) = \frac{D + D^\top}{2}\geq \frac{1}{1+|B|^2}I_3
\]
is uniformly positive definite.
\begin{rema}
  Let us denote by $E = -\nabla_x\phi$ the electric field.
  In the strong magnetic field limit, namely $|B|\rightarrow\infty$, the particles follow the guiding center dynamic (see for example \cite{northrop_1961_guiding, miyamoto_2006_plasma, bellan_2006_fundamentals} in the physics literature).
  They are mainly advected along the magnetic field and transport in the perpendicular plane, called the electric drift, occurs at the next order in $1/|B|$. 
  The corresponding parallel and drift velocities are respectively
  \[
    \mathbf{v}_\| = -\frac{(E\cdot B)B}{|B|^2},\quad \mathbf{v}_\text{drift} = \frac{E\wedge B}{|B|^2}.
  \]
  The guiding center dynamic can be recovered from the limit electron equation in \eqref{driftdiff}. Indeed, note that applying the diffusion matrix to the electric field gives
   \begin{equation}
   \begin{aligned}
   &DE &=&&& E - \frac{1}{1+|B|^2} E\wedge B + \frac{1}{1+|B|^2} (E\wedge B)\wedge B,\\
      &&=&&& \frac{1}{1+|B|^2}\left[(E\cdot B)B - E\wedge B + E\right].\\
   \end{aligned}
  \label{devlim}
   \end{equation}
   Therefore the limit equation for the electrons \emph{formally} rewrites 
   \[
    \partial_t n_e + \nabla_x\cdot\left[(\mathbf{v}_\| + \mathbf{v}_\text{drift} + O\left(1/|B|^2\right)) n_e - D\nabla_x n_e\right]= 0.
   \]
\end{rema}

\begin{defi}
 A triplet $(f_i, n_e, \phi)$ is called a weak solution on $[0,T)$ of the Cauchy problem \eqref{driftdiff} if it satisfies
 \begin{enumerate}
  \item[\textbf{1)}] $f_i\in L^\infty(0,T;L^1_{x,v}\cap L^\infty_{x,v})$;  $f_{i}\geq 0$ almost everywhere,
  \item[\textbf{2)}] $n_e\in L^\infty(0,T;L^1_{x})$; $n_e\geq 0$ almost everywhere,
  \item[\textbf{3)}] $\phi = \Phi\ast_x (\int f_i dv - n_e) \in L^\infty(0,T;\dot{H}^1_x)$ and $n_e\nabla_x\phi\in L^1_\text{loc}([0,T)\times\mathbb{R}^6)$,
\end{enumerate}
  and for every $\varphi\in \mathcal{D}([0,T)\times\mathbb{R}^6)$, $\psi\in \mathcal{D}([0,T)\times\mathbb{R}^3)$
 \begin{enumerate}
  \item[\textbf{4)}] The mappings $t\mapsto\iint_{\mathbb{R}^6}\varphi f_i dv dx$ and $t\mapsto\int_{\mathbb{R}^3}\psi n_e dx$ are continuous and \eqref{driftdiff} holds in the sense of distributions.
 \end{enumerate}
 \label{D2}
\end{defi}

Let us state the main result of this paper. We make the following assumptions on the initial data
\begin{align} 
&\forall\alpha\in\{i,e\},\ f^\text{in}_{\alpha}\geq 0\text{ almost everywhere,} \label{H0}\tag{A1}\\
&\sum_{\alpha\in\{i,e\}}\iint_{\mathbb{R}^6}f^\text{in}_\alpha(|x|+|v|^2 +|\ln f^\text{in}_\alpha|)dv dx + \int_{\mathbb{R}^3}\left|\nabla_x\phi^\text{in}\right|^2dx<+\infty,\label{H1}\tag{A2}\\
&f_i^\text{in}\in L^1_{x,v}\cap L^\infty_{x,v}, \quad f_e^\text{in}\in L^1_{x,v},\label{H2}\tag{A3}\\
&\iint f^\text{in}_idv dx = \iint f^\text{in}_e dv dx.\label{H3}\tag{A4}
\end{align}
where $\phi^\text{in} = \Phi\ast_x(\int f^\text{in}_idv - \int f^\text{in}_e dv)$. 
\begin{rema}
The case of initial data depending on $\varepsilon$ may be treated in the same way as soon as assumptions \eqref{H0} to \eqref{H3} are satisfied uniformly in $\varepsilon$. 
In this case one may pick as limit initial condition for \eqref{driftdiff} any accumulation point of the sequence of initial data. 
\end{rema}
\begin{theo}
 Under assumptions \eqref{H0} to \eqref{H3}, there exists a solution $(f^\varepsilon_i, f^\varepsilon_e, \phi^\varepsilon)$ of \eqref{Vi-VFPe-P} in the sense of Definition \ref{D1}. 
 Moreover, for any such solution, one has, when $\varepsilon\rightarrow 0$ and up to the extraction of a subsequence
 \[
  \begin{aligned}
   &f^\varepsilon_i\longrightarrow f_i &&\text{weakly-$\star$ in }L^\infty([0,T)\times\mathbb{R}^3\times\mathbb{R}^3),\\
   &f^\varepsilon_e\longrightarrow  n_e M &&\text{strongly in }L^1(0,T;L^1_{x,v}),\\
   &\phi^\varepsilon\longrightarrow \phi &&\text{strongly in }L^2(0,T;\dot{W}^{1,p}_{x})\text{ for } 1\leq p <2,\\
  \end{aligned}
 \]
  and the limit $(f_i,  n_e, \phi)$ is a weak solution of \eqref{driftdiff}. 
 \label{main}
\end{theo}

The parabolic limit for the Vlasov-Poisson-Fokker-Planck system with a given ion background
has been studied in dimension $2$ by Poupaud-Soler in \cite{poupaud_2000_parabolic} and Goudon in \cite{goudon_2005_hydrodynamic}.
Masmoudi and El Ghani generalized these results in any dimension in \cite{el_2010_diffusion}, using DiPerna-Lions renormalized solutions and averaging lemmas. 
The procedure they follow was introduced by Masmoudi and Tayeb to study the diffusion limit of a semi-conductor Boltzmann-Poisson system in \cite{masmoudi_2007_diffusion}. 
Let us point out, here and later, the importance of the latter paper, which provides efficient tools to derive a global in time result for our own problem. 
A recent paper \cite{wu_2014_diffusion} of Wu, Lin and Liu treats with this method the case of a multispecies model where several Vlasov-Fokker-Planck equations are coupled by a Poisson equation on a bounded domain.
Let us also mention \cite{ el_2010_diffusion1}, where the author deals with the case of a self-consistent magnetic field in the context of a Vlasov-Maxwell-Fokker-Planck system.

Here, we are considering a multispecies model with an external magnetic field. This brings some new technical difficulties in the analysis.
First, the coupling with a non-collisional kinetic equation rather than a fixed background of charged particles makes it harder to recover regularity sufficient to give a distributional sense to the limit problem. 
Indeed, the regularity given by Vlasov-Poisson in dimension three for the ions impacts that of the electrons (see Lemma \ref{regrho} and \ref{bootstrap}). 
Besides, by considering a magnetic field of the same order of magnitude than the collisions, it is not clear at all that one can control the fast Larmor oscillations at the limit especially in the three dimensional setting. 
As an example of this difficulty in the non-collisional case of Vlasov-Poisson,
we refer to the paper of Saint-Raymond \cite{saint_2000_gyrokinetic} where only the two-dimensional case can be treated since the dynamic along the magnetic field lines is too fast to be captured at the limit in this scaling. 
Here, we are able to control the whole dynamic thanks to the dissipative effect of the collisions and the orthogonality between the Lorentz force and the velocity field (see Lemma \ref{corestimates} and Proposition \ref{prop1}).
An other proof of the nontrivial effect of the $B$ field can be seen with the diffusion matrix of the limit drift-diffusion equation which is anisotropic, for it contains the magnetic field effects. 
Let us mention that such a model was derived in a linear setting (\textit{i.e.} with an external electric field) by Ben Abdallah and El Hajj in \cite{ben_2008_diffusion} with a linear Boltzmann collision kernel.

In the rest of this paper we detail the proof of Theorem \ref{main}. The outline is as follows. In section \ref{apriori}, we introduce natural estimates associated with \eqref{Vi-VFPe-P}. 
These estimates will be crucial to prove our derivation and the existence of solution which is briefly discussed in Section \ref{existencee}. After proving the required compactness of the family of solutions in Section \ref{compactness},
we will take limits in equations in Section \ref{limit}. Eventually, in Section \ref{regularity}, we shall use the algebraic structure of the limit system to gain some regularity which will allow us
to recover \eqref{driftdiff} in a distributional--rather than renormalized--sense.

\section{A priori estimates}

The study of the asymptotic $\varepsilon\rightarrow 0$ requires estimates that are uniform with respect to $\varepsilon$. 
For our coupled system, the only natural identities providing such bounds are mass estimates (see Lemma \ref{lpest}), free energy and entropy inequalities (see Proposition \ref{freeenerg}). 
Let us introduce kinetic energies associated with each species
\[
K^\varepsilon(t) = K^\varepsilon_i(t) + K^\varepsilon_e(t)\text{ where }K_{\alpha} = \frac{1}{2}\iint_{\mathbb{R}^6}|v|^2f^\varepsilon_{\alpha} dv dx \quad\forall\alpha\in\{i,e\}.
\]
The characteristic energy due to electrostatic effects is called electric energy and reads
\[
 E^\varepsilon(t) = \frac{1}{2}\int_{\mathbb{R}^3}|\nabla_x\phi^\varepsilon|^2 dx = \frac{1}{2}\int_{\mathbb{R}^3}\phi^\varepsilon n^\varepsilon dx,
\]
where the last equality stems from the Poisson equation. 
Let us also define the entropy of each species 
\[
S_{\alpha}^\varepsilon(t) = \iint_{\mathbb{R}^6} f_{\alpha}^\varepsilon\ln(f_{\alpha}^\varepsilon) dv dx \quad\forall\alpha\in\{i,e\}.
\]
The natural energy associated with Vlasov-Fokker-Planck type equations is called the free energy and writes, for our system
\[
 U^\varepsilon(t) = E^\varepsilon(t) + K^\varepsilon(t) + S_e^\varepsilon(t).
\]
We also introduce the free energy dissipation given by the following non-negative quantity
\[
 D^\varepsilon(t) = \frac{1}{\varepsilon^2}\iint_{\mathbb{R}^6} \frac{1}{f^\varepsilon_e}|vf^\varepsilon_e+\nabla_vf^\varepsilon_e|^2 dv dx = \frac{4}{\varepsilon^2}\iint_{\mathbb{R}^6} \left|\nabla_v\sqrt{\frac{f_e^\varepsilon}{M}}\right|^2Mdv dx
\]

\label{apriori}
\begin{prop}[Free energy and entropy estimates]
Suppose that $(f_i^\varepsilon, f_e^\varepsilon, \phi^\varepsilon)$ is a smooth localized solution of \eqref{Vi-VFPe-P}. One has the following ''entropy'' estimates, for all $\varepsilon>0$ and $t\in[0,T)$.
\begin{itemize}
 \item The free energy satisfies
 \[
  U^\varepsilon(t) + \int_0^tD^\varepsilon(s)ds = U(0).
 \]
 \item The ion entropy satisfies 
 \[
  S^\varepsilon_i(t) = S^\varepsilon_i(0)
 \]
 \end{itemize}
 \label{freeenerg}
\end{prop}
\begin{proof}
Multiplying the first two equations of \eqref{Vi-VFPe-P} by $|v|^2/2$ and $|v|^2/(2\varepsilon)$ respectively, integrating in $v,x$ and summing the two equations yields, up to an integration by parts,
\begin{equation}
  \dfrac{d}{dt}K^\varepsilon + \int\nabla_x\phi^\varepsilon\cdot j^\varepsilon dx = -\frac{1}{\varepsilon^2}\iint v\cdot(vf_e^\varepsilon + \nabla_vf^\varepsilon_e)dv dx
  \label{2}
\end{equation}

The continuity equation is obtained by integrating the fist two equations of \eqref{Vi-VFPe-P} with respect to $v,x$ and summing the resulting equations after dividing the electron equation by $\varepsilon$.
It reads
\begin{equation}
 \partial_t n^\varepsilon + \nabla_x\cdot j^\varepsilon = 0
 \label{CE2}
\end{equation}
Then using \eqref{CE2} and the Poisson equation, we can rewrite the second term  in \eqref{2}
\begin{equation}
 \int\nabla_x\phi^\varepsilon\cdot j^\varepsilon dx = \frac{1}{2}\dfrac{d}{dt}\int|\nabla_x\phi^\varepsilon|^2dx\\
 \label{3}
\end{equation}
Hence we get an energy estimate from \eqref{2} and \eqref{3}, namely
\begin{equation}
\frac{d}{dt}(K^\varepsilon + E^\varepsilon) = -\frac{1}{\varepsilon^2}\iint v\cdot(vf_e^\varepsilon + \nabla_vf^\varepsilon_e)dv dx
\label{energ}
\end{equation}

The entropy equations are obtained by multiplying the first two equations of \eqref{Vi-VFPe-P} by $\ln f^\varepsilon_i + 1 $ and $\ln f^\varepsilon_e + 1$ respectively and integrating in $x,v$. It yields
\begin{align}
 &\dfrac{d}{dt}S_i^\varepsilon = 0,\label{4a}\\
 &\varepsilon\dfrac{d}{dt}S_e^\varepsilon = -\frac{1}{\varepsilon}\iint \frac{\nabla_vf^\varepsilon_e}{f^\varepsilon_e}\cdot (vf_e^\varepsilon + \nabla_vf_e^\varepsilon) dv dx\label{4b},
\end{align}
where we performed an integration by part of the right-hand side of the second equation. Equation \eqref{4a} provides the ion entropy estimate. 
 By summing \eqref{4b} divided by $\varepsilon$ and \eqref{energ}, we obtain the announced estimate up to an integration in time.
\end{proof}

The other type of natural a priori estimate for Vlasov-type equations is the conservation of $L^p$ norms.
\begin{lem}[$L^p$ norms]
Suppose that $(f_i^\varepsilon, f_e^\varepsilon, \phi^\varepsilon)$ is a smooth localized solution of \eqref{Vi-VFPe-P}. For all $\varepsilon>0$ and $t\in[0,T)$,
\begin{itemize}
\item The distribution function of electrons satisfies
\[
 \|f^\varepsilon_e(t,\cdot,\cdot)\|_{L^1_{x,v}} = \|f^\text{in}_e\|_{L^1_{x,v}} 
\]
\item The distribution function of ions satisfies
\[
 \|f^\varepsilon_i(t,\cdot,\cdot)\|_{L^p_{x,v}} = \|f^\text{in}_i\|_{L^p_{x,v}}, \quad \forall p\in[1,\infty]
\]
\end{itemize}
\label{lpest}
\end{lem}
\begin{proof}
Keeping in mind that distribution functions are positive, the integration of the Vlasov-Fokker-Planck equation in \eqref{Vi-VFPe-P} with respect to $x$ and $v$ provides the first estimate. 
Let us multiply the ion equation of \eqref{Vi-VFPe-P} by $\left({f^\varepsilon_i}\right)^{p-1}/p$ and integrate in $x$ and $v$ to get \[\dfrac{d}{dt}\|f^\varepsilon_i\|_{L^p_{x,v}}^p = 0.\]
Hence $\|f^\varepsilon_i\|_{L^p_{x,v}}$ is constant. Letting $p$ go to infinity gives the limit case. 
\end{proof}

As we are working on an unbounded domain in space, we need to control space moments of the distribution functions to ensure that no mass can be ''lost'' at infinity.
It reduces to controlling current densities as shows the following estimate. 

\begin{lem}[First moment in space]
Suppose that $(f_i^\varepsilon, f_e^\varepsilon, \phi^\varepsilon)$ is a smooth localized solution of \eqref{Vi-VFPe-P}. For all $\varepsilon>0$, $t\in[0,T)$ and $\alpha\in\{i,e\}$
\[
 \iint_{\mathbb{R}^6} |x|f^\varepsilon_{\alpha} dv dx = \int_0^t\int_{\mathbb{R}^3}\frac{x}{|x|}\cdot j^\varepsilon_{\alpha} dx + \iint_{\mathbb{R}^6} |x|f^\text{in}_{\alpha} dv dx
\]
\end{lem}
\begin{proof}
 Multiply the first two equations of \eqref{Vi-VFPe-P} by $|x|$ and integrate in $x$, $v$ and $t$ to obtain the result.
 
\end{proof}

\section{Existence of solutions and uniform in $\varepsilon$ estimates}
\label{existencee}

In this section, we give an existence result for \eqref{Vi-VFPe-P}. 
The a priori estimates of the previous section are necessary to build these solutions, by a mollification procedure. 
Let us mention that this result follows from single-species cases. 
Indeed the coupling between the kinetic equations of \eqref{Vi-VFPe-P} is weak in the sense that, because of the form of the Poisson equation, 
it is possible to isolate the contribution of each species in the electric field $\nabla_x\phi^\varepsilon$. 
The addition of the magnetic field term only cause minor and harmless modifications to usual proofs as it is linear and does not alter a priori estimates.
For the Vlasov-Poisson part, the theory of Arsenev may be applied. We refer to \cite[Theorem 1.3 and 1.4]{bouchut_2000_kinetic} for details. 
The Vlasov-Poisson-Fokker-Planck part of \eqref{Vi-VFPe-P} may be handled with the DiPerna-Lions theory of renormalized solutions \cite{diperna_1988_global, di_1988_solutions}. 
Some technical details may also be found in \cite{bouchut_1995_on} and \cite{masmoudi_2007_diffusion, mischler_2000_on, mischler_2010_kinetic} on bounded domains.
\begin{prop}
 Under assumptions \eqref{H0} to \eqref{H3}, the system \eqref{Vi-VFPe-P} admits a solution in the sense of Definition \ref{D1}. 
 In particular, the solution satisfies the continuity equations in the sense of distributions on $[0,T)\times\mathbb{R}^6$
 \begin{equation}
  \partial_t n^\varepsilon_{\alpha} +  \nabla_x\cdot j^\varepsilon_{\alpha} = 0,\quad\forall\alpha\in\{i,e\}
  \label{CE}
 \end{equation}
 Moreover, the following estimates hold, uniformly in $t\in[0,T)$, $\varepsilon>0$ and $\alpha\in\{i,e\}$,
 \begin{align}
  &U^\varepsilon(t) + \int_0^t D^\varepsilon(s)ds \leq U^\text{in},\label{E1}\\
  &S_i^\varepsilon(t) \leq S_i^\text{in},\label{E1bis}\\
  &\|f^\varepsilon_e(t,\cdot,\cdot)\|_{L^1_{x,v}} \leq \|f^\text{in}_e\|_{L^1_{x,v}},\label{E2}\\
  &\forall  p\in[1,\infty], \quad\|f^\varepsilon_e(t,\cdot,\cdot)\|_{L^p_{x,v}} \leq \|f^\text{in}_e\|_{L^p_{x,v}},\label{E3}\\
  &\iint_{\mathbb{R}^6} |x|f^\varepsilon_{\alpha} dv dx \leq \int_0^t\int_{\mathbb{R}^3}\frac{x}{|x|}\cdot j^\varepsilon_{\alpha} dx + \iint_{\mathbb{R}^6} |x|f^\text{in}_{\alpha} dv dx.\label{Ex}
 \end{align}
 and the distribution functions are non-negative almost everywhere.
 \label{existence}
\end{prop}

From estimates obtained in Proposition \ref{existence} we infer uniform in $\varepsilon$ estimates that will allow us to take limits in the following sections.
Let us first give a name to the particular solutions satisfying these estimates. 
\begin{defi}
Any triplet $(f^\varepsilon_i, f^\varepsilon_e, \phi^\varepsilon)$ which is a solution of the system \eqref{Vi-VFPe-P} in the sense of Definition \ref{D1}, 
associated with initial datum satisfying \eqref{H0} to \eqref{H3}, and itself satisfying estimates of Proposition \ref{existence} is called from now on a \emph{physical solution of \eqref{Vi-VFPe-P}}.
\end{defi}
\begin{prop}
A physical solution of \eqref{Vi-VFPe-P} satisfies the following properties
\begin{enumerate}
 \item[\textbf{(a)}] Control of current densities:
 \begin{align}
  &\|j_i^\varepsilon\|_{L^1_{x}}\leq C(T) + K_i^\varepsilon,\label{estimji}\\ &\|j_e^\varepsilon\|_{L^1_{x}}\leq C(T) + \frac{1}{2}D^\varepsilon.\label{estimje}
 \end{align}
\end{enumerate}
\begin{enumerate}
 \item[\textbf{(b)}] Uniform bounds on the free energy and moments:
 \begin{equation}\sum_{\alpha\in\{i,e\}}\iint_{\mathbb{R}^6}f^\varepsilon_\alpha(|x|+|v|^2 +|\ln f^\varepsilon_\alpha|)dv dx\ + \int_{\mathbb{R}^3}|\nabla_x\phi^\varepsilon|^2 dx + \int_0^tD^\varepsilon(s)ds \leq C(T)
  \label{E4}
 \end{equation}

 \item[\textbf{(c)}] Consequences of (b):
 \begin{equation}
 \begin{aligned}
  &\|j_i^\varepsilon\|_{L^1(0,T;L^1_{x})}\ +\  \|j_e^\varepsilon\|_{L^1(0,T;L^1_{x})}\ +\ 
  \|\nabla_v\sqrt{f^\varepsilon_e} \|_{L^2(0,T;L^2_{x,v})}\\& + \frac{1}{\varepsilon}\|(v\wedge B)\cdot\nabla_vf^\varepsilon_e \|_{L^1(0,T;L^1_{x,v})}\leq C(T)
  \end{aligned}
  \label{E5}
 \end{equation}
\label{estimates}
\end{enumerate}
for some constant $C(T)$ independent of $\varepsilon$ and $t\in[0,T)$ 
\label{corestimates}
\end{prop}
\begin{proof}
  \textbf{(a)} First we can control the ion current density by decomposing the velocity space in the following way,
  \[
  \begin{aligned}
   &\|j_i^\varepsilon\|_{L^1_{x}}&\leq&&& \iint_{|v|<2}|v|f_i^\varepsilon dv dx + \iint_{|v|\geq 2}|v|f_i^\varepsilon dv dx\\
   &&\leq&&& 2\|f^\text{in}_i\|_{L^\infty(0,T;L^1_{x,v})} + \frac{1}{2}\iint_{\mathbb{R}^6}|v|^2f_i^\varepsilon dv dx.
  \end{aligned}
  \]
   We can conclude with \eqref{H1}. 
  
  The electron current density can be controlled by the free energy dissipation. Indeed, following the idea in \cite[Equation 2.21]{poupaud_2000_parabolic}, it writes 
  \[
  j_e = \frac{1}{\varepsilon}\int (v\sqrt{f^\varepsilon_e} + 2 \nabla_v\sqrt{f^\varepsilon_e})\sqrt{f^\varepsilon_e} dv
  \] 
  and thus
   \[
   \|j_e^\varepsilon\|_{L^1_{x}}\leq \frac{1}{2}\|f^\varepsilon_e\|_{L^1_{x,v}} + \frac{1}{2\varepsilon^2}\iint \frac{1}{f^\varepsilon_e}|vf^\varepsilon_e+\nabla_vf^\varepsilon_e|^2 dv dx
  \]
  and the desired result follows using \eqref{H1}.
  
  \medskip\noindent\textbf{(b)} The key arguments here are the entropy estimates \eqref{E1} and \eqref{E1bis}. 
  Since distribution functions are non-negative, we decompose the free energy in the following form
  \begin{equation}
  U^\varepsilon = K^\varepsilon+ E^\varepsilon + S_{e,+}^\varepsilon - S_{e,-}^\varepsilon,
  \label{decomp}
  \end{equation}
  where we define, for $\alpha\in\{i,e\}$,
  \[
   S_{\alpha,+}^\varepsilon  = \iint f_\alpha^\varepsilon \ln^+f_\alpha^\varepsilon dv dx,
  \]
  \[
   S_{\alpha,-}^\varepsilon  = \iint f_\alpha^\varepsilon \ln^-f_\alpha^\varepsilon dv dx,
  \]
  with $\ln^+(s) = \max\{\ln(s), 0\}$ and $\ln^-(s) = \max\{-\ln(s), 0\}$ for $s>0$. 
  By applying the same arguments as in \cite[Lemma 2.3]{poupaud_2000_parabolic} and estimate \eqref{Ex}, one can get the following bound on the negative part of the entropy
    \begin{equation}
    S_{e,-}^\varepsilon\leq C + \frac{1}{2}\left(K_e^\varepsilon + \|j_e^\varepsilon\|_{L^1(0,t,L^1_x)}\right) + \frac{1}{2}\iint|x|f^\text{in}_e dv dx,
    \label{lowerentropy}
  \end{equation}
  for some positive constant $C$.
  By inequality \eqref{estimje} on the electron current density and \eqref{lowerentropy}
  \[
   S_{e,-}^\varepsilon\leq \frac{1}{2}\left(K_e^\varepsilon + \int_0^tD^\varepsilon(s)ds\right)+C(T).
  \]
  Now with estimate \eqref{E1} and the decomposition of the free energy \eqref{decomp}, one can conclude that $K^\varepsilon$, $E^\varepsilon$, $\int_0^tD^\varepsilon(s)ds$, $S_{e,-}^\varepsilon$ and $S_{e,+}^\varepsilon$ are uniformly bounded in $\varepsilon$ and $t$. 
  Replacing the index $e$ by $i$, inequality \eqref{lowerentropy} holds true for the ion related quantities and the bound \eqref{estimji} on the ion current density and estimate \eqref{E1bis} on the ions entropy give the boundedness of $S_{i,-}^\varepsilon$ and $S_{i,+}^\varepsilon$.

  \medskip\noindent\textbf{(c)} The estimates on current densities follow using (a) and (b). 
  Now,  as in \cite[Corollary 5.3]{el_2010_diffusion}, the following computation
  \[
  \begin{aligned}
   &\|\nabla_v\sqrt{f^\varepsilon_e} \|_{L^2(0,T;L^2_{x,v})}^2 &=&&& \int_0^T\iint\Big{(}\frac{1}{4}|v\sqrt{f^\varepsilon_e} + 2\nabla_v\sqrt{f_e^\varepsilon}|^2 
   - \frac{1}{4}|v|^2f_e^\varepsilon - v\sqrt{f^\varepsilon_e}\cdot\nabla_v\sqrt{f_e^\varepsilon}\Big{)} dvdxdt 
   \\&&\leq&&&\frac{\varepsilon^2}{4}D^\varepsilon + \frac 12\int_0^T\iint (\nabla_v\cdot v) f_e^\varepsilon dv dx dt\\
   &&\leq&&& \frac{\varepsilon^2}{4}D^\varepsilon + \frac{3T}{2}\iint f^\text{in}_e dv dx,
  \end{aligned}
  \]
  provides the third bound.
  Finally we can control the term related to the magnetic field by noticing that, since $(v\wedge B)\cdot v = 0$,
  \[
   \begin{aligned}
   &\frac{1}{\varepsilon}\|(v\wedge B)\cdot\nabla_vf^\varepsilon_e \|_{L^1(0,T;L^1_{x,v})}&=&&&\frac{1}{\varepsilon}\int_0^T\iint \left|(v\wedge B)\cdot(2\nabla_v\sqrt{f^\varepsilon_e} + v \sqrt{f^\varepsilon_e})\sqrt{f^\varepsilon_e}\right| dx dv dt\\
   &&\leq&&&\|B\|_{L^\infty_{t,x}}\left(\frac{1}{\varepsilon^2}\int_0^T D^\varepsilon dt + T \sup_{[0,T)}K^\varepsilon_e\right).
   \end{aligned}
  \]
  We conclude with the estimates in (b).
\end{proof}
\begin{rema}
  Actually, the estimate on the ion current density can be largely improved using a classical moment lemma \cite[Lemma 3.1]{golse_1999_vlasov}. The latter gives that $(j_i^\varepsilon)_\varepsilon$ is uniformly bounded in $L^\infty(0,T; L^{5/4}_x)$.
\end{rema}

\section{Compactness of the family of solutions}
\label{compactness}
From the estimates of the foregoing section, we can infer some compactness. When $\varepsilon$ tends to $0$, estimates \eqref{E3} and \eqref{E4} give, up to the extraction of a subsequence,
\begin{equation}
 f_i^\varepsilon\rightarrow f_i \text{ weakly-$\star$ in }L^\infty(0,T;L^p(\mathbb{R}^6))\text{ for }p\in(1,\infty],
 \label{limfi}
\end{equation}
and
\begin{equation}
 \nabla_x\phi^\varepsilon\rightarrow \nabla_x\phi \text{ weakly-$\star$ in }L^\infty(0,T;L^2(\mathbb{R}^3)),
 \label{cvphi}
\end{equation}
where the limit is a gradient because it is irrotational in the sense of distributions.
\begin{lem}
 Families of physical solutions of \eqref{Vi-VFPe-P} satisfy the following properties
 \begin{enumerate}
  \item[\textbf{(a)}] $\left(f^\varepsilon_i\right)_\varepsilon$ and $\left(f^\varepsilon_e\right)_\varepsilon$ are weakly relatively compact in $L^1([0,T)\times\mathbb{R}^6)$.
  \item[\textbf{(b)}] $\left( n^\varepsilon_i\right)_\varepsilon$ and $\left( n^\varepsilon_e\right)_\varepsilon$ are weakly relatively compact in $L^1([0,T)\times\mathbb{R}^3)$.
  \item[\textbf{(c)}] $f_e^\varepsilon -  n^\varepsilon_e M\rightarrow 0$ in $L^1([0,T)\times\mathbb{R}^6)$ when $\varepsilon\rightarrow 0$
 \end{enumerate}
 \label{lemma1}
\end{lem}
\begin{proof}
The first two assertions follow from the Dunford-Pettis theorem.

\medskip\noindent\textbf{(a, b)} Let us define $W: x \mapsto e^{-|x|}/(8\pi)$, $\mathcal{M}: (x,v)\mapsto W(x)M(v)$ and $\Psi: u\geq0\mapsto u(\ln^+ u +1)$. 
 Note that $W$ and $\mathcal{M}$ are of integral $1$ and $\Psi$ is a convex, non-negative, increasing function satisfying
 \[
  \lim_{u\rightarrow \infty}\frac{\Psi(u)}{u} = +\infty.
 \]
 Let $\alpha\in\{i,e\}$. One sees that by the Jensen inequality
 \[
  \begin{aligned}
  &\int\Psi(\frac{n_{\alpha}^\varepsilon}{W})Wdx
  &=&&& \int\Psi\left(\int\frac{f_{\alpha}^\varepsilon}{\mathcal{M}}M dv\right)W dx\\
  &&\leq&&&\iint\Psi\left(\frac{f_{\alpha}^\varepsilon}{\mathcal{M}}\right)\mathcal{M} dv dx\\
  &&=&&&\iint f_{\alpha}^\varepsilon\ln\left(\frac{f_{\alpha}^\varepsilon}{\mathcal{M}}\right) dv dx\\
  &&\leq&&& S_{\alpha,+}^\varepsilon + S_{\alpha,-}^\varepsilon + K_{\alpha}^\varepsilon + \iint |x|f_{\alpha} dv dx + C\iint f_{\alpha}^\varepsilon dv dx,
  \end{aligned}
 \]
 for some constant $C>0$. Because of the uniform estimates \eqref{E3} and \eqref{E4}, we hence get by the de la Vallée Poussin lemma the equi-integrability of the bounded families 
 \[\{f_\alpha^\varepsilon/\mathcal{M}\}_\varepsilon\subset L^1(0,T;L^1(\mathcal{M} dv dx))\] and \[\{n_\alpha^\varepsilon/W\}_\varepsilon\subset L^1(0,T;L^1(W dx)).\] 
 This yields the announced weak compactness by the Dunford-Pettis theorem.

 \medskip\noindent\textbf{(c)} Using the log-Sobolev inequality (see \cite{gross_1975_logarithmic}), we get an upper bound for the relative entropy of the electrons with respect to a local Maxwellian
 \[
  \begin{aligned}
  &\int_0^t\iint f_e^\varepsilon\ln \left(\frac{f_e^\varepsilon}{ n^\varepsilon M}\right) dv dx ds &\leq&&& 2\int_0^t\iint\left|\nabla_v\sqrt{\frac{f_e^\varepsilon}{M}}\right|^2 M dv dx ds \\
  &&\leq&&& C(T) \varepsilon^2,
  \end{aligned}
 \]
where the last inequality comes from estimate \eqref{E4}. By the Cziszar-Kullback-Pinsker inequality (see \cite[Theorem 3.1 and Section 4]{csiszar_1967_information}), we have
 \[
  \begin{aligned}
  &\int_0^t\left(\iint\left|f^\varepsilon_e -  n^\varepsilon_e M\right| dv dx\right)^2 ds &\leq&&& 2\sup_{[0,T)}\left(\int n^\varepsilon_e dx\right)\int_0^t\iint f_e^\varepsilon\ln \left(\frac{f_e^\varepsilon}{ n^\varepsilon_e M}\right) dv dx ds \\
  &&\leq&&&  C(T) \varepsilon^2,
  \end{aligned}
 \]
 which yields the $L^2(0,T;L^1_{x,v})$ convergence hence the expected result.
\end{proof}

The results of Lemma \ref{lemma1} are not sufficient to take limits in equations. We need to gain strong compactness to deal with non-linear terms. 
While there is no hope of doing so with the distribution functions, the particular averaging properties of Vlasov type equations allows us to get additional compactness results in space on the macroscopic densities. 
Indeed, by the dispersion property of the $v\cdot\nabla_x$ transport operator, one gains regularity on velocity averages of the distribution function.
Besides, time compactness stems from the continuity equations \eqref{CE} and the uniform bounds on current densities. We shall apply the following averaging lemma that is adapted from the famous DiPerna and Lions results in \cite{diperna_1991_lp}.

\begin{lem}
 Let $\left(h^\varepsilon\right)_\varepsilon$ be a bounded sequence in $L^2(0,T;L^2_{x}(L^2_{\text{loc},v}))$ satisfying in the sense of distributions
 \[
  \varepsilon\partial_th^\varepsilon + v\cdot \nabla_xh^\varepsilon = h^\varepsilon_0 + \nabla_v\cdot h^\varepsilon_1
 \]
 where $\left(h^\varepsilon_0\right)_\varepsilon$, $\left(h^\varepsilon_1\right)_\varepsilon$ are bounded sequences in $L^1(0,T;L^1_{x}(L^1_{\text{loc},v}))$. Then for all $\psi\in\mathcal{D}(\mathbb{R}^3)$,
 \[
 \left\|\int (\tau_yh^\varepsilon - h^\varepsilon)\psi dv\right\|_{L^1(0,T;L^1_x)}\rightarrow 0
 \]
 when $y\rightarrow 0$ uniformly in $\varepsilon$, where $\tau_y$ is the translation of vector $y$ in the $x$ variable.
 \label{avlem}
\end{lem}
\begin{proof}
 We refer to \cite[Appendix 2]{masmoudi_2007_diffusion} for a proof.
\end{proof}

Let us highlight that the following result is crucial to the rest of the proof of Theorem \ref{main} and that in particular, point (b) in Proposition \ref{prop1} depends upon the control of the magnetic leading term of the equation by the entropy dissipation.
\begin{prop}
 Families of physical solutions of \eqref{Vi-VFPe-P} satisfy the following strong compactness properties
 \begin{enumerate}
  \item[\textbf{(a)}] $\left(\nabla_x\phi^\varepsilon\right)_\varepsilon$ is relatively compact in $L^2(0,T;L^p_{\text{loc},x})$ for $1\leq p<2$,
  \item[\textbf{(b)}] $\left( n_e^\varepsilon\right)_\varepsilon$ is relatively compact in $L^1([0,T)\times\mathbb{R}^3)$.
 \end{enumerate}
 \label{prop1}
\end{prop}
 \begin{proof} 
 The result of (a) stems from applying the Aubin-Lions-Simon lemma \cite[Theorem 5, Corollary 4]{simon_1987_compact} to the family of electric fields and the proof can be readily adapted from \cite[Proposition 3.3 3)]{masmoudi_2007_diffusion}.

  \medskip\noindent\item\textbf{(b)} Let us define \[\beta_\delta:u\mapsto \frac{u}{1+\delta u}.\] 
  There exits $C_\delta>0$ such that, for all $u\geq0$
  \[
   \beta_\delta(u)\leq\min(\frac{1}{\delta}, u),\quad (1+\sqrt{u})|\beta_\delta'(u)|\leq 1+\frac{1}{\sqrt{\delta}},\quad |u\beta_\delta''(u)|\leq \frac 12. 
  \]

   In particular, one checks that $\beta_\delta$ satisfies the requirements of a renormalization function for the Vlasov-Fokker-Planck equation and hence \eqref{ebeta} holds. Now set 
  \[
   \left\{
    \begin{aligned}
     &h^\varepsilon = \beta_\delta(f_e^\varepsilon)\\
     &h^\varepsilon_0 = -\nabla_vf^\varepsilon_e\cdot\frac{vf^\varepsilon_e + \nabla_vf^\varepsilon_e}{\varepsilon}\beta_\delta''(f^\varepsilon_e) + \frac{1}{\varepsilon}(v\wedge B)\cdot\nabla_v(\beta_\delta(f^\varepsilon_e))\\
     &h^\varepsilon_1 = -\nabla_x\phi^\varepsilon\beta_\delta(f^\varepsilon_e) + \frac{vf^\varepsilon_e + \nabla_vf^\varepsilon_e}{\varepsilon}\beta_\delta'(f^\varepsilon_e)
    \end{aligned}
   \right.
  \]
   With the estimates of Proposition \ref{estimates}, one can show that the sequences $\left(h^\varepsilon\right)_\varepsilon$, $\left(h^\varepsilon_0\right)_\varepsilon$ and $\left(h^\varepsilon_1\right)_\varepsilon$ satisfy the hypotheses of Lemma \ref{avlem}.
   We refer to \cite[Proposition 6.1]{el_2010_diffusion} for the details. In our case, we additionally need to check that the magnetic field term satisfies the $L^1$ bound. 
   Indeed, using that $(v\wedge B)\cdot v = 0$, we have
 \[
 \begin{aligned}
  &&&\left\|\frac{1}{\varepsilon}(v\wedge B)\cdot\nabla_v(\beta_\delta(f^\varepsilon_e))\right\|_{L^1(0,T;L^1_{x,v})}\\
  &&=&\left\|(v\wedge B)\sqrt{f^\varepsilon_e}\cdot\frac{1}{\varepsilon}(2\nabla_v\sqrt{f^\varepsilon_e} + v\sqrt{f^\varepsilon_e}) \beta_\delta'(f^\varepsilon_e)\right\|_{L^1(0,T;L^1_{x,v})}\\
  &&\leq&\sqrt{T}\left\|B\right\|_{L^\infty(0,T;L^\infty_x)}\left\||v|^2f^\varepsilon_e\right\|_{L^\infty(0,T;L^1_{x,v})}\left\|\frac{vf^\varepsilon_e + \nabla_vf^\varepsilon_e}{\varepsilon\sqrt{f^\varepsilon_e}}\right\|_{L^2(0,T;L^2_{x,v})}\left\|\beta_\delta'(f^\varepsilon_e)\right\|_{L^\infty(0,T;L^\infty_{x,v})}.
 \end{aligned}
 \]
Then for any fixed $\delta$ and any $\psi\in \mathcal{D}(\mathbb{R}^3)$
 \begin{equation}
 \left\|\int (\tau_y\beta_\delta(f_e^\varepsilon) - \beta_\delta(f_e^\varepsilon))\psi dv\right\|_{L^1(0,T;L^1_x)}\rightarrow 0
 \label{translabeta}
 \end{equation}
 when $y\rightarrow 0$ uniformly in $\varepsilon$. The $L^\infty(0,T;L^1((1+|v|^2)dv dx)$ uniform bound on $f_e^\varepsilon$ from \eqref{E2} and \eqref{E4} allows us to extend this to $\psi(v)\equiv 1$.
 Now, we can also take the limit $\delta\rightarrow 0$ uniformly in $\varepsilon$, using the equi-integrability of the family $(f_e^\varepsilon)_\varepsilon$. Indeed, since, for any $u>0$
 \[
  0\leq\beta_\delta(u)\leq u\ \text{ and }\ |\beta_\delta(u)-u|\leq \delta u^2, 
 \]
 one has
 \begin{equation}
 \|f_e^\varepsilon - \beta_\delta(f_e^\varepsilon)\|_{L^1([0,T)\times\mathbb{R}^6}\leq 2\int_0^T\iint_{f_e^\varepsilon>M} f_e^\varepsilon dv dx dt 
 +\delta\int_0^T\iint_{\{f_e^\varepsilon<M\}} |f_e^\varepsilon|^2 dv dx dt
  \label{approxbeta}
 \end{equation}
 for an arbitrary $M>0$. The first term is $O(1/\ln(|M|))$ because of the uniform bound on the entropy $S_e^\varepsilon$. 
 Using the uniform bound on the mass, the second term is seen to be $O(\delta M)$. Take $M = \delta^{-1/2}$ to conclude. 
 Therefore, using \eqref{translabeta} and \eqref{approxbeta}, one has
 \[
 \left\|\tau_y n_e^\varepsilon -  n_e^\varepsilon\right\|_{L^1(0,T;L^1_x)}\rightarrow 0
 \]
  when $y\rightarrow 0$ uniformly in $\varepsilon$. 
  
 Finally, using the continuity equation \eqref{CE}, we get a $L^1(0,T;W^{-1,1}_x)$ bound on $\partial_t n_e^\varepsilon$. With the uniform $L^1((1+|x|)dx dt)$ estimate \eqref{Ex} on $ n_e^\varepsilon$, this gives the relative compactness of the sequence.
\end{proof}

Using the results of Lemma \ref{lemma1} and Proposition \ref{prop1}, we get the following strong convergence results concerning the macroscopic density of electrons.
\begin{lem}
 The families of physical solutions of \eqref{Vi-VFPe-P} satisfy the following properties, up to the extraction of a subsequence. There exists $ n_e\in L^1([0,T)\times\mathbb{R}^3)$ such that when $\varepsilon\rightarrow 0$
 \begin{enumerate}
  \item[\textbf{(a)}] $ n_e^\varepsilon\rightarrow n_e$ in $L^1(0,T;L^1_x)$,
  \item[\textbf{(b)}] $\sqrt{ n_e^\varepsilon}\rightarrow\sqrt{ n_e}$ in $L^2(0,T;L^2_x)$,
  \item[\textbf{(c)}] $f_e^\varepsilon\rightarrow n_eM$ in $L^1(0,T;L^1_{x,v})$,
  \item[\textbf{(d)}] $\theta_{\varepsilon,\lambda} = \sqrt{f_e^\varepsilon + \lambda M}\rightarrow\sqrt{ n_e + \lambda}\sqrt{M}$ in $L^2(0,T;L^2_{x,v})$, 
  \item[\textbf{(e)}] $\frac{M}{\theta_{\varepsilon,\lambda}}\rightarrow \frac{\sqrt{M}}{\sqrt{ n_e + \lambda}}$ in $L^2(0,T;L^2_{x,v})$,
  \item[\textbf{(f)}] $\frac{\sqrt{ n_e^\varepsilon}M}{\theta_{\varepsilon,\lambda}}\rightarrow \frac{\sqrt{ n_e M}}{\sqrt{ n_e + \lambda}}$ in $L^2(0,T;L^2_{x,v})$,
 \end{enumerate}
 for any $\lambda>0$.
 \label{cvrho}
\end{lem}
\begin{proof}
 Properties (a, b, c, d) are straightforward consequences of Lemma~\ref{lemma1}~(c) and Proposition \ref{prop1} (b). 
 The last two assertions follow from (b), (d) and the $L^\infty$ bounds on $\sqrt{M}/\theta_{\varepsilon, \lambda}$ and $\sqrt{n_e}/\sqrt{n_e+\lambda}$.
%
%
\end{proof}

\section{Taking limits in equation \eqref{etheta}}
\label{limit}

Using the previous compactness results, we can readily take limits in weak formulations of the Vlasov equation for ions in \eqref{Vi-VFPe-P}, 
the continuity equation for electrons \eqref{CE}, and the Poisson equation in \eqref{Vi-VFPe-P}.
\begin{lem}
 The limits $n_e\in L^1([0,T)\times\mathbb{R}^3)$, $f_i\in L^\infty(0,T;L^1_{x,v}\cap L^\infty_{x,v})$ and $\nabla_x\phi\in L^\infty(0,T;L^2_{x})$ defined respectively in Lemma \ref{cvrho}, 
 \eqref{limfi} and \eqref{cvphi} satisfy in the sense of distributions,
 \[
 \left\{
 \begin{aligned}
  &\partial_tf_i + v\cdot\nabla_xf_i - \nabla_x\phi\cdot\nabla_vf_i +(v\wedge B)\cdot\nabla_vf_i = 0,\\
  &\partial_t n_e + \nabla_x\cdot j_e = 0,\\
  &-\Delta_x\phi =  n_i -  n_e,\\
  &n_\alpha = \int f_\alpha dv,\ \forall\alpha\in\{i,e\}, 
 \end{aligned}
 \right.
\] 
for any accumulation point $j_e$ of the family $(j_e^\varepsilon)_{\varepsilon}$. 
\end{lem}

At this point, we do not know much on limits\footnote{From estimate \ref{estimates}, we only know \textit{a priori}  that, up to the extraction of a subsequence, 
$(j_e^\varepsilon)_{\varepsilon}$ converges to a finite Radon measure for the weak star topology.} 
of the electron current density $j_e^\varepsilon$. 
Our goal is now to characterize them. Let us introduce,
\begin{equation}
{r^\varepsilon_e} = \frac{1}{\varepsilon\sqrt{M}}(\sqrt{f^\varepsilon_e}-\sqrt{ n_e^\varepsilon M}),
\label{re}
\end{equation}
and
\begin{equation}
{R^\varepsilon_e} = \frac{1}{2\varepsilon}(f^\varepsilon_e-n_e^\varepsilon M)
\label{Re}
\end{equation}
so that the distribution function of electrons may be written as 
\begin{equation}
\begin{aligned}
&f^\varepsilon_e &=&&& n_e^\varepsilon M + 2\varepsilon {r^\varepsilon_e}\sqrt{ n_e^\varepsilon}M + \varepsilon^2|{r^\varepsilon_e}|^2 M\\
&&=&&& n_e^\varepsilon M + 2\varepsilon R^\varepsilon_e.
\end{aligned}
\label{ansatzf}
\end{equation} 
With this in hands, one writes the electron current density as
\[
 j^\varepsilon_e = 2\sqrt{ n^\varepsilon_e}\int {r^\varepsilon_e} M v dv + \varepsilon\int |{r^\varepsilon_e}|^2M v dv
\]

As in \cite{el_2010_diffusion, el_2010_diffusion1}, we aim at taking the limit $\varepsilon\rightarrow 0$ on the latter and characterize the limit current.
We first gather some useful estimates on ${r^\varepsilon_e}$.
\begin{prop}
 Let $(f_e^\varepsilon)_\varepsilon$ be the electron distribution functions of a family of physical solutions of \eqref{Vi-VFPe-P} and define $r^\varepsilon_e$ by \eqref{re}. Then, the following uniform estimates hold 
 \begin{enumerate}
  \item[\textbf{(a)}] $({r^\varepsilon_e})_\varepsilon$ is bounded in $L^2(0,T;L^2(Mdv dx))$,
  \item[\textbf{(b)}] $(\varepsilon|{r^\varepsilon_e}|^2|v|^2M)_\varepsilon$ is bounded in $L^1(0,T;L^1_{x,v})$,
  \item[\textbf{(c)}] $(\sqrt{\varepsilon}|{r^\varepsilon_e}|^2|v|M)_\varepsilon$ is bounded in $L^1(0,T;L^1_{x,v})$,
  \item[\textbf{(d)}] $(\nabla_vr^\varepsilon_e)_\varepsilon$ is bounded in $L^2(0,T;L^2(Mdv dx))$.
 \end{enumerate}
 \label{reps}
\end{prop}

\begin{proof}
Properties (a, b, c) can be readily adapted from \cite[Proposition 5.5]{el_2010_diffusion}.

  \medskip\noindent\textbf{(d)} From the definition of $r_e^\varepsilon$ stems 
  \[
   \int_0^T\iint |\nabla_vr^\varepsilon_e|^2Mdvdxdt =  \frac{1}{\varepsilon^2}\int_0^T\iint \left|\nabla_v\sqrt{\frac{f_e^\varepsilon}{M}}\right|^2Mdvdxdt.
  \]
  The right-hand side is uniformly bounded thanks to estimate \eqref{E4} on the entropy dissipation.

\end{proof}

From now on we consider a subsequence of a family of physical solutions of \eqref{Vi-VFPe-P} such that all convergence properties following from the previous compactness results hold.
Let us denote by $r_e$ the weak $L^2(0,T;L^2(Mdv dx))$ limit of $\left({r^\varepsilon_e}\right)_\varepsilon$. 
\begin{lem}
 When $\varepsilon\rightarrow 0$, the family $\left(j_e^\varepsilon\right)_\varepsilon$ satisfies
 \[
  j_e^\varepsilon\rightarrow j_e:=2\sqrt{ n_e}\int r_e v M dv\text{ weakly in }L^1(0,T;L^1_{x}).
 \]
 \label{limitj}
\end{lem}
\begin{proof}
  Take the limit in the following expression, coming from \eqref{ansatzf},
  \begin{equation}
  j^\varepsilon_e = 2\sqrt{ n^\varepsilon_e}\int v{r^\varepsilon_e} M  dv + \varepsilon\int v |{r^\varepsilon_e}|^2M  dv.
  \label{ansatz}
  \end{equation}
  The first term in the right-hand side converges weakly in $L^1(0,T;L^1_{x,v})$. Indeed,  
  \[\sqrt{ n_e^\varepsilon}\rightarrow\sqrt{ n_e}\text{ strongly in }L^2(0,T;L^2_x),\] and since ${r^\varepsilon_e}\rightarrow r_e $ weakly in $L^2(0,T;L^2(Mdv dx))$ and 
   $(t,v)\mapsto v\in L^2(0,T;L^2(Mdv))$ \[\int v r_e^\varepsilon M dv\rightarrow \int v r_e  M dv\text{ weakly in }L^2(0,T;L^2_{x}).\]
   The second term in the right-hand side of \eqref{ansatz} goes to $0$ in $L^1(0,T;L^1_x)$ since we know from  Proposition \ref{reps} (c) that $(\sqrt{\varepsilon}|{r^\varepsilon_e}|^2|v|M)_\varepsilon$ is bounded in $L^1(0,T;L^1_{x,v})$. 
\end{proof}

Now we focus on the limit of equation \eqref{etheta} from Definition \ref{D1} to derive the last pieces of information we need to characterize the limit of the current density $j_e$.
To do so, we introduce elements of notation associated with the operator $\mathcal{L}_A$. We denote by $L^2_M$ the Hilbert space $L^2(\mathbb{R}^3, M^{-1}dv)$ endowed with the scalar product
\[
 \left\langle f,g\right\rangle = \int_{\mathbb{R}^3}fg M^{-1}dv.
\]
Almost everywhere in $t,x$, \[\mathcal{L}_A : f \mapsto \nabla_v\cdot(A(t,x) v f + \nabla_v f) =  \nabla_v\cdot(vf +\nabla_v f) + (v\wedge B(t,x))\cdot\nabla_vf\] is an unbounded operator on $L^2_M$. 
On
\begin{equation}
  \mathcal{L}_A(f) = \nabla_v\cdot(M\nabla_v\left(\frac fM\right)) + (v\wedge B)\cdot\nabla_v\left(\frac{f}{\sqrt{M}}\right)\sqrt{M},
  \label{LA}
\end{equation}
one sees that the formal adjoint of $\mathcal{L}_A$ in $L^2_M$ is given by
\begin{equation}
  \mathcal{L}_A^*(f)= \nabla_v\cdot(vf +\nabla_v f) - (v\wedge B)\cdot\nabla_vf =  \nabla_v\cdot(A^\top v f + \nabla_v f) = \mathcal{L}_{A^\top}(f).
  \label{LA*}
\end{equation}
\begin{prop}
 The limit density $n_e$ and electric field $-\nabla_x\phi$ are such that $(\nabla_x\sqrt{ n_e} - \frac{1}{2}\nabla_x\phi\sqrt{ n_e})\in L^2(0,T;L^2_x)$.
 Moreover the limit electron current density $j_e$ satisfies in the sense of distributions
 \[
 j_e = -2\sqrt{ n_e}A^{-1}(\nabla_x\sqrt{ n_e} - \frac{1}{2}\nabla_x\phi\sqrt{ n_e}),
 \]
\end{prop}
\begin{proof}
 We know from Lemma \ref{limitj} that \[j_e = 2\sqrt{ n_e}\int r_e v M dv.\] We will now take the limit in the renormalized equation \eqref{etheta} in order to get an additional equation characterizing $r_e$.
 Let us recall that \eqref{etheta} reads
  \[
   \varepsilon\partial_t\theta_{\varepsilon,\lambda} + v\cdot\nabla_x\theta_{\varepsilon,\lambda} + \nabla_x\phi^\varepsilon\cdot\nabla_v\theta_{\varepsilon,\lambda}  = \frac{1}{2\varepsilon\theta_{\varepsilon,\lambda}}\mathcal{L}_{A}(f^\varepsilon_e) -
   \frac{\lambda M}{2\theta_{\varepsilon,\lambda}}v\cdot\nabla_x\phi^\varepsilon.
  \]
 By the strong convergence of $\theta_{\varepsilon,\lambda}$ from Lemma \ref{cvrho} (d) and the weak convergence of $\nabla_x\phi^\varepsilon$ from \eqref{cvphi}, the left-hand side of equation \eqref{etheta} converges to
 \[
 \begin{aligned}
  &\nabla_x\cdot(v\sqrt{( n_e+\lambda)M}) + \nabla_v\cdot(\nabla_x\phi\sqrt{( n_e+\lambda)M})\\ &= v\sqrt{M}\cdot(\nabla_x\sqrt{ n_e+\lambda} - \frac{1}{2}\nabla_x\phi\sqrt{ n_e+\lambda}).
 \end{aligned}
 \]
 
 The first term of the right-hand side of the renormalized equation \eqref{etheta} may be written, using \eqref{Re} and the fact that $M\in\mathrm{Ker}\mathcal{L}_A$, as
 \begin{equation}
   \frac{1}{2\varepsilon\theta_{\varepsilon,\lambda}}\mathcal{L}_{A}(f^\varepsilon_e) = \frac{1}{\theta_{\varepsilon,\lambda}}\mathcal{L}_{A}(R_e^\varepsilon).
   \label{firstterm} 
 \end{equation} 
 Now, for any $\varphi\in\mathcal{D}([0,T)\times\mathbb{R}^6)$, one has
 \[
  I^\varepsilon=\int_0^T\iint\frac{1}{\theta_{\varepsilon,\lambda}}\mathcal{L}_{A}(R_e^\varepsilon)\varphi dv dx dt := I^\varepsilon_1 + I^\varepsilon_2 + I^\varepsilon_3.
 \]
 with
 \[
  \begin{aligned}
   &I^\varepsilon_1:= \int_0^T\iint\frac{\varphi M}{\theta_{\varepsilon,\lambda}} (v\wedge B)\cdot\nabla_v\left(\frac{R_e^\varepsilon}{M}\right) dv dx dt,\\
   &I^\varepsilon_2:= - \int_0^T\iint\frac{\sqrt{M}}{\theta_{\varepsilon,\lambda}} \nabla_v\left(\frac{R_e^\varepsilon}{M}\right)\cdot\nabla_v\left(\frac{\varphi}{\sqrt{M}}\right) M dv dx dt,\\
   &I^\varepsilon_3:= - \int_0^T\iint \nabla_v\left(\frac{R_e^\varepsilon}{M}\right)\cdot\nabla_v\left(\frac{\sqrt{M}}{\theta_{\varepsilon,\lambda}}\right)\varphi\sqrt{M} dv dx dt.
  \end{aligned}
 \]
 where we have performed an integration by part
on the Fokker-Planck part of the operator $\mathcal{L}_A$ defined in \eqref{LA} to obtain $I^\varepsilon_2$ and $I^\varepsilon_3$. Mark that 
 \[
  \nabla_v\left(\frac{\sqrt{M}}{\theta_{\varepsilon,\lambda}}\right) = -\frac{1}{2}\frac{\nabla_v\left(\frac{f_e^\varepsilon}{M}\right)}{\left(\frac{f_e^\varepsilon}{M} + \lambda\right)^{3/2}}
 \]
 so that using the definition of $R_e^\varepsilon$ in \eqref{Re}, one has, for some constant $C(T)$,
 \[
  |I^\varepsilon_3|\leq \frac{1}{\varepsilon}\int_0^T\iint M \left|\nabla_v\sqrt{\frac{f_e^\varepsilon}{M}}\right|^2 dv dx dt \left\|\frac{f_e^\varepsilon}{\theta_{\varepsilon,\lambda}^2}\right\|_{L^{\infty}(0,T;L^\infty_{x,v})}\left\|\frac{\varphi}{\theta_{\varepsilon,\lambda}}\right\|_{L^{\infty}(0,T;L^\infty_{x,v})}
  \leq C(T)\varepsilon,
 \]
by the control on the free energy dissipation provided by estimate \eqref{E4}. Hence $I^\varepsilon_3$ goes to $0$. To handle $I^\varepsilon_1$ and $I^\varepsilon_2$, 
we decompose $R_e^\varepsilon$ according to,
 \[
  \frac{R_e^\varepsilon}{M} = \sqrt{n_e^\varepsilon}r_e^\varepsilon + \frac{\varepsilon}{2}|r_e^\varepsilon|^2
 \]
 Doing so, the contribution of the magnetic field becomes 
 \[
 \begin{aligned}
  &I^\varepsilon_1 &=& \int_0^T\iint\left(\frac{\sqrt{n_e^\varepsilon}M}{\theta_{\varepsilon,\lambda}} - \frac{\sqrt{n_e M}}{\sqrt{n_e+\lambda}}\right) \varphi (v\wedge B)\cdot\nabla_vr_e^\varepsilon dv dx dt \\
  &&& -\int_0^T\iint\frac{\sqrt{n_e}}{\sqrt{n_e+\lambda}}  r_e^\varepsilon\sqrt{M} (v\wedge B)\cdot\nabla_v\varphi dv dx dt \\
  &&&+ \varepsilon\int_0^T\iint\frac{\varphi}{\theta_{\varepsilon,\lambda}}(v\wedge B)\cdot\nabla_vr_e^\varepsilon\sqrt{M}r_e^\varepsilon\sqrt{M} dv dx dt  .
 \end{aligned}
 \]
 after an integration by parts in the second term. The first and third terms of the right-hand side go to zero by Lemma \ref{cvrho} (f) and \ref{reps} (a),(d). Therefore
 \[
  \begin{aligned}
  &I^\varepsilon_1&\longrightarrow& -\int_0^T\iint\frac{\sqrt{n_e}}{\sqrt{n_e+\lambda}}  r_e\sqrt{M} (v\wedge B)\cdot\nabla_v\varphi dv dx dt\\
  &&&= \int_0^T\iint\frac{\sqrt{n_e}}{\sqrt{n_e+\lambda}\sqrt{M}}  (v\wedge B)\cdot\nabla_v(r_eM) \varphi dv dx dt
  \end{aligned}
 \]
  The contribution of the Fokker-Planck part may be written as 
 \[
 \begin{aligned}
  &I^\varepsilon_2 &=& -\int_0^T\iint\left(\frac{\sqrt{n_e^\varepsilon}M}{\theta_{\varepsilon,\lambda}} - \frac{\sqrt{n_e M}}{\sqrt{n_e+\lambda}}\right)\sqrt{M}\nabla_vr_e^\varepsilon\cdot\nabla_v\left(\frac{\varphi}{\sqrt{M}}\right) dv dx dt \\
  &&&+ \int_0^T\iint\frac{\sqrt{n_e}}{\sqrt{n_e+\lambda}}  r_e^\varepsilon \nabla_v\cdot\left(M\nabla_v\left(\frac{\varphi}{\sqrt{M}}\right)\right) dv dx dt \\
  &&& -\varepsilon\int_0^T\iint\frac{\sqrt{M}}{\theta_{\varepsilon,\lambda}}\sqrt{M}r_e^\varepsilon\sqrt{M}\nabla_vr_e^\varepsilon\cdot\nabla_v\left(\frac{\varphi}{\sqrt{M}}\right) dv dx dt  .
 \end{aligned}
 \]
 As for $I^\varepsilon_1$, the first and third terms go to $0$ and therefore
 \[
  \begin{aligned}
  &I^\varepsilon_2&\longrightarrow&\int_0^T\iint\frac{\sqrt{n_e}}{\sqrt{n_e+\lambda}\sqrt{M}} \nabla_v\cdot\left(M\nabla_v\left(r_e\right)\right)\varphi dv dx dt 
  \end{aligned}
 \]
 We eventually showed that, in the sense of distributions 
 \begin{equation}
  \frac{1}{2\varepsilon\theta_{\varepsilon,\lambda}}\mathcal{L}_{A}(f^\varepsilon_e)\longrightarrow\frac{\sqrt{n_e}}{\sqrt{n_e+\lambda}\sqrt{M}} \mathcal{L}_{A}(r_eM)
  \label{cvLA}
 \end{equation}
 
 The convergence of the second term of the right-hand side of the renormalized equation \eqref{etheta} stems from
 the strong convergence of $M/\theta_{\varepsilon,\lambda}$ from Lemma \ref{cvrho} and the weak convergence of $\nabla_x\phi^\varepsilon$ from \eqref{cvphi}.
 Finally, we receive for any $\lambda>0$
 \[
 \begin{aligned}
  &v\sqrt{M}\cdot\left(\nabla_x\sqrt{ n_e + \lambda}- \frac{1}{2}\nabla_x\phi\sqrt{ n_e+\lambda}\right)\\
  & = \frac{\sqrt{ n_e}}{\sqrt{( n_e +\lambda)M}}\mathcal{L}_{A}(r_e M) - \frac{\lambda M}{2\sqrt{( n_e +\lambda)M}}v\cdot\nabla_x\phi.
  \end{aligned}
 \]
 By dominated convergence, one may take the limit $\lambda\rightarrow 0$ to obtain
 \[
  vM\cdot (\nabla_x\sqrt{ n_e} - \frac{1}{2}\nabla_x\phi\sqrt{ n_e}) = \mathcal{L}_{A}(r_e M)
 \]
 in the sense of distributions. Since the left-hand side of the former equality is rapidly decaying in $v$, 
 one may actually multiply the previous equation by $v$ and integrate in the $v$ variable to derive in the sense of distributions
 \[
 \begin{aligned}
  &(\nabla_x\sqrt{ n_e} - \frac{1}{2}\nabla_x\phi\sqrt{ n_e}) &=&&& \int v\mathcal{L}_{A}(r_e M)dv\\
  &&=&&& \int \mathcal{L}_{A}^*(v M) r_e dv\\
  &&=&&& \int \left[\mathcal{L}_{I}(v M) - (v\wedge B)\cdot\nabla_v(vM)\right] r_e dv\\
  &&=&&& \int -(v+(v\wedge B)) r_e M dv = - A\int r_evM dv,\\
  \end{aligned}
 \]
  Since $A(t,x)$ is invertible one gets the result by combining this identity with the expression of $j_e$ from Lemma \ref{limitj}.
\end{proof}

\section{Regularity of the limit}
\label{regularity}
Let us summarize what we have proved so far. The triplet $(f_i, \nabla_x\phi,  n_e)\in L^\infty(0,T;L^1_{x,v}\cap L^\infty_{x,v})\times L^\infty(0,T;L^2_{x})\times L^1(0,T;L^1_{x})$ is such that in the sense of distributions 
\[
 \left\{
 \begin{aligned}
  &\partial_tf_i + v\cdot\nabla_xf_i - \nabla_x\phi\cdot\nabla_vf_i + (v\wedge B)\cdot\nabla_vf_i = 0,\\
  &\partial_t n_e + \nabla_x\cdot j_e = 0,\\
  &(\nabla_x\sqrt{ n_e} - \frac{1}{2}\nabla_x\phi\sqrt{ n_e})\in L^2(0,T,L^2_x)\\
  & j_e = -2 \sqrt{ n_e}A^{-1}(\nabla_x\sqrt{ n_e} - \frac{1}{2}\nabla_x\phi\sqrt{ n_e})\\
  &-\Delta_x\phi =  n_i -  n_e,
 \end{aligned}
 \right.
\]
with initial data
\[
 f_i(0,\cdot,\cdot) = f_i^\text{in}\text{ and } n_e(0,\cdot,\cdot) = \int f_e^\text{in}dv =:  n_e^\text{in}.
\]
Note that because of uniform bounds on $(\partial_t n_e^\varepsilon)_\varepsilon$ in $L^1(0,T;W^{-1,1}_x)$ and on $(\partial_tf_i^\varepsilon)_\varepsilon$ in $L^1(0,T;W^{-1,1}_{x,v})$, 
we get that, by the Arzela-Ascoli theorem, the limit functions $t\mapsto\int n_e\varphi dx$ and $t\mapsto\int f_i\psi dv dx$ are continuous on $[0,T)$ for any test functions $\varphi$ and $\psi$.
Thus, we do recover the above initial conditions for the limit system. On the other hand, usual arguments based on the convexity and lower semi-continuity of the energy and entropy functionals and corresponding
uniform estimates prove the boundedness of the following quantities, uniformly in $t\in[0,T)$
\[
\begin{aligned}
 &\int  n_e|\ln  n_e| dx + \int|x| n_e dx + \iint(|v|^2 + |x|) f_i dv dx\leq C(T) \\
 &\iint f_i|\ln f_i| dv dx\leq \|f_i^\text{in}\ln f_i^\text{in}\|_{L^1_{x,v}}\\ 
\end{aligned}
\]
 for some positive constant $C(T)$. Furthermore, the following mass estimates and global neutrality result hold uniformly in $t\in[0,T)$
  \begin{align}
  &\| n_e\|_{L^1_{x}} = \| f_i\|_{L^1_{x,v}} = \|f_e^\text{in}\|_{L^1_{x,v}} = \|f_i^\text{in}\|_{L^1_{x,v}}.\label{globalneutr}\\
  &\|f_i\|_{L^p_{x,v}}\leq \|f_i^\text{in}\|_{L^p_{x,v}}\text{for $p\in(1,+\infty]$}\label{lpineq}
  \end{align}

Equality is verified for the mass estimate thanks to the tightness of the distribution functions that comes from the control of space and velocity moments.

The classical moment lemma \cite[Lemma 3.1]{golse_1999_vlasov} shows that, by the boundedness of $f_i$ in $L^\infty(0,T;L^1((1+|v|^2)dvdx)\cap L^\infty_{x,v})$, the limit macroscopic ion density has the regularity
\begin{equation}
 n_i\in L^\infty(0,T;L^{5/3}_x).
\end{equation}

We may gain some additional regularity on $ n_e$ using the particular structure of the limit system. The procedure we set up is a generalization of Lemma 7.1 in \cite{masmoudi_2007_diffusion}. 
Here the situation is trickier because the ion background $n_i$ is not regular enough to reach directly an $L^2$ regularity in space for $n_e$ in order for the product $n_e\nabla_x\phi$ to make sense. 
However one may first gain some regularity and conclude by a bootstrap argument in Lemma \ref{bootstrap}.
\begin{lem}
  Let $n_e\in L^1([0,T)\times\mathbb{R}^d)$ be a non-negative function that satisfies 
 \begin{align}
  &\nabla_x\sqrt{ n_e} + \frac{1}{2}E\sqrt{ n_e} = G,\label{*}\\
  &\nabla_x\cdot E =  n_i -  n_e\label{**}
 \end{align}
 in the sense of distributions, where $G\in L^2([0,T)\times\mathbb{R}^d)$, $E\in L^2([0,T)\times\mathbb{R}^d)$ and $n_i\in \mathcal{D}'([0,T)\times\mathbb{R}^d)$. Then, for any $p\in[1, 2]$ it holds
 \[
   n_i\in  L^{p}([0,T)\times\mathbb{R}^d)\Longrightarrow n_e\in  L^{p}([0,T)\times\mathbb{R}^d)
 \]
 \label{regrho}
\end{lem}

\begin{proof}
Let $p\in(1,2]$. The first step of the proof is the renormalization of equation \eqref{*}. 
We define hereafter the particular renormalization function we use and which is built to recover in the end an $L^{p}$ bound on $n_e$. 
Let us define $\gamma\in\mathcal{C}^\infty(\mathbb{R}_+)$ such that $\gamma(s) = s$  on [0,1],  $\gamma(s) = 2$ for $s>3$ and $0\leq\gamma'\leq 1$. 
 Now set, for $\delta\in(0,1]$,
 \[
 \gamma_\delta(s) = \frac{1}{p-1}\left(\frac{1}{\delta}\gamma(\delta s) + 1\right)^{p-1}\quad\text{and}\quad\gamma_\delta'(s) = \gamma'(\delta s)\left(\frac{1}{\delta}\gamma(\delta s) + 1\right)^{p-2}.
 \]
The derivative of the renormalization function satisfies
  \begin{equation}
  \left|\gamma_{\delta}'(s) \right|\leq 1\quad\text{and}\quad \left|s\gamma_{\delta}'(s) \right|\leq\frac{3}{\delta}.
   \label{estimgammaprim}
  \end{equation}
Equation \eqref{*} implies that $\nabla_x\sqrt{ n_e}\in L^1_\text{loc}([0,T)\times\mathbb{R}^d)$. Let us renormalize equation \eqref{*} by multiplying it by $\gamma_{\delta}'(\sqrt{ n_e})$
\[
\nabla_x\gamma_{\delta}(\sqrt{ n_e}) + \frac{1}{2}E\sqrt{ n_e}\gamma_{\delta}'(\sqrt{ n_e}) = G\gamma_{\delta}'(\sqrt{ n_e}).
\]
 One can check with \eqref{estimgammaprim} that every term is square integrable.
 By taking the $L^2$ norm of the equation and expanding we obtain
\begin{equation}
\begin{aligned}
 &\|\nabla_x\gamma_{\delta}(\sqrt{ n_e})\|_{L^2}^2 + \frac{1}{4}\|E\sqrt{ n_e}\gamma_{\delta}'(\sqrt{ n_e}) \|_{L^2}^2 \\
 &+ \iint\nabla_x\gamma_{\delta}(\sqrt{ n_e}) \cdot E\sqrt{ n_e}\gamma_{\delta}'(\sqrt{ n_e}) dx dt\leq \|G\|_{L^2}^2. 
\end{aligned}
 \label{abc}
\end{equation}
We want to rewrite the third term as the scalar product of $E$ with a gradient in order to use \eqref{**}. Let us define
 \begin{equation}
  \tilde{\gamma}_{\delta}(s) = \int_0^s (\gamma_{\delta}'(u))^2udu.
  \label{formulgamma}
 \end{equation}
Using \eqref{formulgamma} in the third term of \eqref{abc} and dropping the first two non-negative terms yields
\[
 \iint E\cdot\nabla_x{\tilde{\gamma}_{\delta}(\sqrt{ n_e})} dx dt \leq \|G\|_{L^2}^2.
\]
By using equation \eqref{**}, one gets after integrating by parts
\begin{equation}
\iint\tilde{\gamma}_{\delta}(\sqrt{ n_e})\left( n_e- n_i\right)\leq\|G\|_{L^2}^2.
 \label{***}
\end{equation}
Let us estimate $\tilde{\gamma}_{\delta}(\sqrt{ n_e})$. For $s\in[0,1/\delta]$, one obtains  
 \[
\tilde{\gamma}_\delta(s) =\int_0^su(u+1)^{2p-4}du\leq \int_0^su^{2p-3}du = \frac{s^{2p-2}}{2p-2},
\]
and
\[
\tilde{\gamma}_\delta(s) =\left\{
\begin{aligned}
& \frac{(s+1)^{2p-2}-1}{2p-2}-\frac{(s+1)^{2p-3}-1}{2p-3}&\text{if }p \neq \frac{3}{2},\\
&s - \ln(s+1) &\text{if }p = \frac{3}{2}.
\end{aligned}
\right.
\]
Then, one readily checks that for any $p\in(1,2]$ there exists $C_1>0$ such that for any $\delta\in(0,1]$ and $s<\frac{1}{\delta}$ it holds
\[
 \tilde{\gamma}_\delta(s)\geq  \frac{1}{4}s^{2p-2} - C_1
\]
For $s>1/\delta$, there exists $C_2>0$ depending only on $p$ such that
 \[
 \tilde{\gamma}_\delta(s) \leq \tilde{\gamma}_\delta(3/\delta) \leq \left(\frac{1}{\delta}\right)^{2p-2}\int_0^{3} \gamma(u)^{2p-4}u du = C_2 s^{2p-2}
 \]
 This provides the following estimates, for any $s\geq 0$ and $C = \max(C_1, C_2, 1/(2p-2))$
 \begin{equation}
   \left(\frac{1}{4}s^{2p-2} - C\right)\mathds{1}_{\{s\leq 1/\delta\}}(s)\leq\tilde{\gamma}_{\delta}(s)\leq C s^{2p-2},\quad \forall s\geq0.
   \label{****}
 \end{equation}
 We now use estimate \eqref{****} in \eqref{***}. First we get rid of the part involving the ion density $ n_i$. Using Young's inequality, for any $\eta>0$ there exists some constant $C_\eta>0$ such that
 \[
 \begin{aligned}
 &\iint\tilde{\gamma}_{\delta}(\sqrt{ n_e}) n_i &\leq&&& C_\eta \| n_i\|^{p}_{L^{p}} + \eta \iint \left(\tilde{\gamma}_{\delta}(\sqrt{ n_e})\right)^{\frac{p}{p-1}} dx dt.\\
 &&\leq&&& C_\eta \| n_i\|^{p}_{L^{p}} + C^{\frac{1}{p-1}}\eta \iint  n_e \tilde{\gamma}_{\delta}(\sqrt{ n_e}) dx dt
 \end{aligned}
 \]
 where we used estimate \eqref{****} in the second inequality. On the other hand, we have
 \[
 \iint  n_e \tilde{\gamma}_{\delta}(\sqrt{ n_e}) dx dt\geq\frac{1}{4}\iint_{\{ n_e\leq 1/\delta^2\}}  n_e^p dx dt - C\|n_e\|_{L^1}.
 \]
 As a result, by taking $\eta$ sufficiently small, one gets 
 \[0\leq \iint n_e^{p}\mathds{1}_{\{ n_e\leq 1/\delta^2\}} dx dt\leq C(\|G\|_{L^2},\| n_i\|_{L^{p}},\| n_e\|_{L^1}),\]
 uniformly in $\delta$. Taking the monotone limit $\delta\rightarrow 0$ concludes the proof. 

\end{proof}

The regularity of $n_e$ that one may establish by the previous proof is limited by the regularity of $n_i$. 
Nevertheless the available regularity of $n_i$ is sufficient to provide us with a termwise sense for $j_e$.
\begin{lem}
 Limiting densities of families of physical solutions satisfy $\nabla_x\phi\sqrt{ n_e}\in L^1(0,T;L^2_{\text{loc},x})$, $\sqrt{ n_e}\in L^1(0,T;H^1_{\text{loc},x})$ and in the sense of distributions
  \[
   j_e = -D(\nabla_x n_e - \nabla_x\phi n_e),
  \]
  where we recall that $D = A^{-1}$.
  \label{bootstrap}
\end{lem}
\begin{proof}
First by an application of Lemma \ref{regrho} with $p=5/3$, we show that the source $n$ in the Poisson equation $-\Delta_x\phi =  n_i -  n_e =:  n$ lies in $L^{5/3}(0,T;L^{5/3}_x)\cap L^{\infty}(0,T;L^{1}_x)$ and therefore,
with an $L^\infty(0,T; L^2_x)$ electric field, it yields $\nabla_x\phi = \nabla_x\Phi\ast_xn$. Thus, by the Hardy-Littlewood-Sobolev inequality, $\nabla_x\phi\in L^{5/3}(0,T;L^{15/4}_x)$. 
Hence, by the Hölder inequality, the product $\nabla_x\phi\sqrt{ n_e}$ is in $L^{10/9}(0,T;L^{30/17}_x)$ and since 
\[
 \nabla_x\sqrt{ n_e} + \frac{1}{2}E\sqrt{ n_e}\in L^2([0,T)\times\mathbb{R}^3),
\]
this yields $\sqrt{ n_e}\in L^{10/9}(0,T;W^{1,30/17}_{\text{loc},x})$.
By Sobolev embedding this gives at least $\sqrt{ n_e}\in L^{1}(0,T;L^4_{\text{loc},x})$ and 
since $\nabla_x\phi\in L^{\infty}(0,T;L^2_{x})$, the product $\nabla_x\phi \sqrt{n_e}$ belongs to $L^1(0,T;L^2_{\text{loc},x})$ which yields the results.
\end{proof}

This completes the proof of Theorem \ref{main}.

\medskip
\noindent \textsc{Acknowledgements.} The author would like to thank Francis Filbet and Luis Miguel Rodrigues both for their support and many insightful comments.

\bibliographystyle{plain}
\bibliography{bibli}

\end{document}